\documentclass[11pt]{article}
\usepackage{epic,graphicx,amssymb,amsmath,amsthm}

\newtheorem{theorem}{Theorem}[section]
\newtheorem{lemma}[theorem]{Lemma}
\newtheorem{proposition}[theorem]{Proposition}
\newtheorem{corollary}[theorem]{Corollary}
\theoremstyle{remark}
\newtheorem{remark}[theorem]{Remark}

\newtheorem{example}[theorem]{Example}
\newtheorem{definition}[theorem]{Definition}

\usepackage[utf8]{inputenc}
\usepackage{epic,graphicx,amssymb,amsmath}
\usepackage{comment}

\def\N{\mathbb N}
\def\Z{\mathbb Z}
\def\R{\mathbb R}
\def\Q{\mathbb Q}

\def\A{\mathcal A}
\def\B{\mathcal B}
\def\a{\mathfrak A}
\def\b{\mathfrak B}
\def\c{\mathfrak C}
\def\d{\mathfrak D}

\def\pfz{\begin{proof}}
\def\pfk{\end{proof}}
\def\pfkNoQed{\end{proof}}

\numberwithin{theorem}{section}
\newcommand{\bA}{\begin{gathered}\boxed{\:\begin{gathered}0 \\ 0\end{gathered}\:}\\A\end{gathered}}
\newcommand{\bC}{\begin{gathered}\boxed{\:\begin{gathered}1 \\ 1\end{gathered}\:}\\C\end{gathered}}
\newcommand{\bAC}{\begin{gathered}\boxed{\:\begin{gathered}0\ 1 \\ 1\ 0\end{gathered}\:}\\B\end{gathered}}

\usepackage{enumerate}
\usepackage{tikz}
\usetikzlibrary{calc,arrows,positioning}

\begin{document}
\pagestyle{plain}

\title{Exchange of three intervals: itineraries, substitutions and palindromicity}
\author{Zuzana Mas\'akov\'a \and Edita Pelantov\'a \and \v St\v ep\'an Starosta}
\maketitle              % typeset the title of the contribution

\begin{abstract}
Given a symmetric exchange of three intervals, we provide a detailed description of the return times to a subinterval and the corresponding itineraries.
We apply our results to morphisms fixing words coding non-degenerate three interval exchange transformation. This allows us to prove that the conjecture stated by Hof, Knill and Simon is valid for such infinite words.
\end{abstract}

%%%%%%%%%%%%%%%%%%%%%%%%%%%%%%%%%%%%%%%%%%%%%%%%%%%%%%%%%%%%%%%%%%%%%%%%%%
%%%%%%%%%%%%%%%%%%%%%%%%%%%%%%%%%%%%%%%%%%%%%%%%%%%%%%%%%%%%%%%%%%%%%%%%%%
%%%%%%%%%%%%%%%%%%%%%%%%%%%%%%%%%%%%%%%%%%%%%%%%%%%%%%%%%%%%%%%%%%%%%%%%%%
%%%%%%%%%%%%%%%%%%%%%%%%%%%%%%%%%%%%%%%%%%%%%%%%%%%%%%%%%%%%%%%%%%%%%%%%%%
\section{Introduction}

Interval exchange transformations have been extensively studied since the works on their ergodic aspects by Sinai~\cite{Sinai}, Keane~\cite{Keane}, Veech~\cite{Veech2}, Rauzy~\cite{Rauzy}, and others. For an overview, see~\cite{Veech} and references therein.
Among general dynamical systems, interval exchanges have the interesting property that the Poincar\'e first return map  is again a mapping of the same type, i.e.\ an exchange of (possibly different number of) intervals.
Rauzy~\cite{Rauzy2} used this fact to present a generalization of the classical continued fractions expansion.

It is commonly known that interval exchange transformations provide a very useful framework for the study of infinite words arising by coding of rotations, in particular Sturmian words. These are usually defined as aperiodic infinite words with lowest factor complexity. Equivalently, one obtains Sturmian words by binary coding of the trajectory under exchange $T$ of two intervals $[0,\alpha)$, $[\alpha,1)$ with $\alpha$ irrational. Given a subinterval $I\subset[0,1)$, the first return map $T_I$ to $I$ is an exchange of at most three intervals, although the return itineraries of points can take up to four values. The set of these itineraries can be used to describe certain characteristics of Sturmian words, namely the return words, see~\cite{VuillonSturm}, or abelian return words~\cite{RiSaVa}, and invariance under morphisms~\cite{Yasutomi}. %(NEBO SMIM UVEST NASE?, CO BERTHE, DELALI TO PRES INDUKCI?)

Infinite words coding exchange of $k$-intervals, $k\geq 3$, are also in focus for several decades~\cite{FeHoZa,ferenczi} and, here too, one finds a close relation between their combinatorial features and the properties of the induced map, see for example~\cite{vuillon} for a result on return words or~\cite{Adamczewski} about substitutivity of interval exchange words.
A generalized version of the Poincar\'e first return map was used in~\cite{BlBrLaVu11} for description of palindromic complexity in codings of rotations.
These words are in intimate relation with three interval exchange words.

In this paper we focus on codings of a non-degenerate symmetric exchange $T:J\to J$ of three intervals. First we describe the  return times to a general interval $I\subset J$ and provide an insight on the structure of the set of $I$-itineraries. These results are given as Theorem~\ref{t:returntime} and then interpreted as analogues of the well known three gap and three distance theorems.

A particular attention is paid to the special cases when the set of $I$-itineraries has only three elements. These cases belong to the most interesting from the combinatorial point of view, since they provide information about return words to factors, and about the morphisms preserving three interval exchange words. For mutually conjugated morphisms, we describe in Theorem~\ref{t:interceptconjug} the relation between intercepts of their fixed points, as was done for Sturmian morphisms in~\cite{Peng}.
We also show that morphisms conjugated only to themselves do not have a non-degenerated fixed point.

The most important application of our results is a contribution to the solution of the question stated by Hof, Knill and Simon~\cite{HoKnSi} for palindromic words.
We refer to it as the HKS conjecture and adopt its reformulation by Tan~\cite{BoTan} who showed its validity for binary words. Labb\'e~\cite{La2013} presented a counterexample for the conjecture on ternary alphabet; the ternary word not satisfying the hypothesis turns out to be a degenerate three interval exchange word. In fact, degenerate three interval exchange words are just morphic images of binary, in fact Sturmian, words.
In this paper we show that for non-degenerated words coding exchange of three intervals the HKS conjecture holds. Let us mention that the latter result has been announced at the DLT conference~\cite{DLT}. Here we provide a full proof.

This paper is organized as follows.
Section~\ref{sec:prelim} contains the necessary notions from combinatorics on words.
Symmetric $k$-interval exchange transformations and their properties with respect to the first return map are treated in Section~\ref{sec:iet}.
Section~\ref{sec:3iet} focuses on specific properties when $k=3$. The main theorem about return times in three interval exchanges is given in
Section~\ref{sec:returntimes}.
In Section~\ref{sec:Gap} we put our results into context of three gaps and three distance theorems.

The specific case when the set of $I$-itineraries has only three elements is studied in Section~\ref{sec:3}. This allows us to describe the return words to palindromic bispecial factors.
In Section~\ref{sec:substitutions} we focus on substitution invariance of words coding interval exchange transformations.
The key lemma for the demonstration of our Theorem~\ref{thm:hks} on HKS conjecture requires some knowledge about
the relation of substitutions fixing words coding three interval exchange and Sturmian morphisms. This topic
is treated in Section~\ref{sec:ternarizace}. The proof of Theorem~\ref{thm:hks} is then provided in Section~\ref{sec:HKS} together with some other comments.

%%%%%%%%%%%%%%%%%%%%%%%%%%%%%%%%%%%%%%%%%%%%

%%%%%%%%%%%%%%%%%%%%%%%%%%%%%%%%%%%%%%%%%%%%%%%%%%%%%%%%%%%%%%%%%%%%%%%%%%
\section{Preliminaries}\label{sec:prelim}

Let us recall necessary notions and notation from combinatorics on words. For a basic overview we
refer to~\cite{Lo2}. An \textit{alphabet} is a finite set of symbols, called \textit{letters}.
A \textit{finite word} $w$ over an alphabet $\A$ of length $|w|=n$ is a concatenation $w=w_0\cdots w_{n-1}$
of letters $w_i\in\A$. The set of all finite words over $\A$ equipped with the operation of concatenation
and the empty word $\epsilon$ is a monoid denoted by $\A^*$. For a fixed letter $a\in\A$, the number of
occurrences of $a$ in $w$, i.e., the number of indices $i$ such that $w_i=a$, is denoted by $|w|_a$.
 The \textit{reversal} or \textit{mirror image} of the word $w$ is the word $\overline{w}=w_{n-1}\cdots w_0$.
  A word $w$ for which $w=\overline{w}$ is called a \textit{palindrome}.
An \textit{infinite word} ${\bf u}$ is an infinite concatenation ${\bf u}=u_0u_1u_2\ldots \in\A^\N$. An
infinite word ${\bf u}=wvvv\ldots$ with $w,v\in\A^*$ is said to be \textit{eventually periodic}; it is
said to be \textit{aperiodic} if it is not of such form. We say that $w\in\A^*$ is a \textit{factor} of
$v\in\A^*\cup\A^\N$ if $v=w'ww''$ for some $w'\in\A^*$ and $w''\in\A^*\cup\A^\N$. If $w'=\epsilon$ or
 $w''=\epsilon$, then $w$ is a \textit{prefix} or \textit{suffix} of $v$, respectively.
If $v=wu$, then we write $u=w^{-1}v$ and $w=vu^{-1}$.

The set ${\mathcal L}({\bf u})$ of all finite factors of an infinite word ${\bf u}$ is called the
\textit{language of ${\bf u}$}.   If for any  factor $w\in {\mathcal L}({\bf u})$
there exist at least two indices $i$ such that
$w$ is a prefix of the infinite word $u_iu_{i+1}u_{i+2}\cdots$, the word  ${\bf u}$ is {\it recurrent}.
Given a factor $w\in {\mathcal L}({\bf u})$, a finite word $v$ such that $vw$ belongs
  to ${\mathcal L}({\bf u})$ and the word $w$ occurs in $vw$ exactly twice, once as a prefix and once as a suffix of
   $vw$, is called a \textit{return word of $w$}. If any factor $w$ of an infinite recurrent word  ${\bf u}$ has only finitely many return words, the word ${\bf u}$
   is called {\it uniformly recurrent}.

The \textit{factor complexity} ${\mathcal C}_{\bf u}$ is the function $\N\to\N$ counting the number of factors of ${\bf u}$ of length $n$.
It is known that the factor complexity of an aperiodic infinite word ${\bf u}$ satisfies ${\mathcal C}_{\bf u}(n)\geq n+1$ for
all $n$. Aperiodic infinite words having the minimal complexity ${\mathcal C}_{\bf u}(n)= n+1$ for all $n$ are called
\textit{Sturmian words}. Since ${\mathcal C}_{\bf u}(1)=2$, they are binary words.
Sturmian words can be equivalently defined in many different frameworks, one of them is coding of an exchange of two
intervals.

Let $\A$ and $\B$ be alphabets. Let $\varphi:\A^*\to\B^*$ be a \textit{morphism}, i.e., $\varphi(wv)=\varphi(w)\varphi(v)$
 for all $w,v\in\A^*$. We say that $\varphi$ is \textit{non-erasing} if $\varphi(b)\neq \epsilon$ for every $b\in\A$.
 The action of $\varphi$ can be naturally extended to infinite words ${\bf u}\in\A^\N$ by setting
 $\varphi({\bf u})=\varphi(u_0)\varphi(u_1)\varphi(u_2)\ldots$. If $\A=\B$ and $\varphi({\bf u})={\bf u}$,
  then ${\bf u}$ is said to be a \textit{fixed point} of $\varphi$.
A non-erasing morphism $\varphi:\A^*\to\A^*$ such that there is a letter $a\in\A$ satisfying
$\varphi(a)=aw$ for some non-empty word $w$ is called a \textit{substitution}. Obviously, a substitution has
always a fixed point, namely $\lim_{n\to\infty}\varphi^n(a)$ where the limit
is taken over the product topology.
An infinite word which is a fixed point of a substitution is called a \textit{pure morphic} word.
Let $\A = \{a_1, \ldots, a_k\}$ and $\B = \{b_1, \ldots , b_\ell \}$.
One associates to every morphism $\varphi: \A \to \B$ its \textit{incidence matrix} $M_\varphi\in\N^{k\times \ell}$ defined by
\[
(M_\varphi)_{ij}=|\varphi(a_i)|_{b_j}\,,\quad \text{ for }1\leq i\leq k,\ 1\leq j\leq \ell.
\]
A morphism $\varphi:\A^*\to\A^*$ is said to be \textit{primitive} if all elements of some power of its
incidence matrix $M_\varphi\in\N^{k\times k}$ are positive.
A specific class of morphisms is formed by the so-called \textit{Sturmian morphisms} which are
 defined over the binary alphabet $\{0,1\}$ and for which there exists a Sturmian word ${\bf u}$ such that $\varphi({\bf u})$ is also Sturmian.
For an overview about properties of Sturmian morphisms see~\cite{Lo2}.

%%%%%%%%%%%%%%%%%%%%%%%%%%%%%%%%%%%%%%%%%%%%%%%%%%%%%%%%%%%%%%%%%%%%%%%%%%
\section{Itineraries in symmetric exchange of intervals}\label{sec:iet}

For disjoint intervals $K$ and $K'$ we write $K<K'$ if for $x\in K$ and $x'\in K'$ we have $x<x'$.
Let $J$ be a semi-closed interval. Consider a partition $J=J_0\cup\dots\cup J_{k-1}$ of $J$ into a
disjoint union of semi-closed subintervals $J_0<J_1<\cdots <J_{k-1}$. A bijection $T:J\to J$ is
called an \textit{exchange of $k$ intervals with permutation $\pi$}
if there exist numbers $c_0,\dots,c_{k-1}$ such that for  $0\leq i<k$ one has
\begin{equation}\label{eq:obecIET}
T(x)=x+c_i \text{ for } x\in J_i\,,
\end{equation}
where $\pi$ is a permutation of $\{0,1,\dots,k-1\}$ such that $T(J_i)<T(J_j)$
for $\pi(i)<\pi(j)$. In other words, the permutation $\pi$ determines the order of intervals $T(J_i)$.
If $\pi$ is the permutation $i\mapsto k-i+1$, then $T$ is called a \textit{symmetric} interval exchange.

The orbit of a given point $\rho$ is the infinite sequence $\rho$,
$T(\rho)$, $T^2(\rho)$, $T^3(\rho)$, \dots.   It  can be coded by an
infinite word ${\bf u}_\rho=u_0u_1u_2\ldots$ over the alphabet $\{0,1,\dots,k-1\}$ given by
\[
u_n=X \quad \text{if } T^n({\rho}) \in J_X \quad \text{ for } X \in \{0,1,\dots,k-1\}.
\]
The point $\rho$ is called the \textit{intercept} of $\bf{u}_\rho$.
An exchange of intervals satisfies the \textit{minimality condition} if the orbit of any given $\rho\in[0,1)$ is
 dense in $J$.
In this case, the word ${\bf u}_\rho$ is aperiodic, uniformly recurrent, and the language of ${\bf u}_\rho$ depends
only on the parameters of the transformation $T$ and not on the intercept $\rho$ itself.
The complexity of an infinite word ${\bf{u}_\rho}$ is known to satisfy ${\mathcal C}_{\bf{u}_\rho}(n)\leq (k-1)n+1$
(see \cite{FeHoZa}). If for every $n\in\N$ we have ${\mathcal C}_{\bf{u}_\rho}(n)= (k-1)n+1$, then the transformation
 $T$ and the word $\bf{u}_\rho$ are said to be \textit{non-degenerate}. A sufficient and necessary condition on $T$ to
  be non-degenerate is that the orbits of the discontinuity points of $T$ are
  infinite and disjoint. This condition is known under the abbreviation i.d.o.c.

\begin{definition}
Let $T$ be an exchange of $k$ intervals satisfying the minimality condition.
Given a subinterval $I\subset J$, we define the mapping $r_I: I\to \Z^+=\{1,2,3,\dots\}$ by
\[
r_I(x)=\min \{n\in\Z^+ \colon T^n(x)\in I\}\,,
\]
the so-called \textit{return time} to $I$. The prefix of length $r_I(x)$ of the word
${\bf u}_x$ coding the orbit of a point $x\in I$
is called the \textit{$I$-itinerary} of $x$ and denoted $R_I(x)$. The set of all $I$-itineraries is
denoted by ${\rm It}_I=\{R_I(x) \colon x\in I\}$.
The mapping $T_I:I\to I$ defined by
\[
T_I(x) = T^{r_I(x)}(x)
\]
is said to be the \textit{first return map of $T$ to $I$}, or \textit{induced map of $T$ on $I$}.
\end{definition}

Throughout the paper, when it is clear from the context, we sometimes omit the index $I$ in $r_I$ or $R_I$.
It is known from Keane~\cite{Keane} that if $T$ is an exchange of $k$ intervals and $I\subset J$,
then $\mathit{It}_I$ has at most $k+2$ elements, and, consequently, $T_I$ is an exchange of at most $k+2$ intervals.

\begin{remark}\label{pozn:skakani}
Let $X\in\{0,1,\ldots, k-1\}$.
If $I\subset J_X$, then $T(I)$ is an interval and we have
\[
R \text{ is an $I$-itinerary }\ \Leftrightarrow \ X^{-1}RX \text{ is a $T(I)$-itinerary}.
\]
Similarly, if $I\subset T(J_X)$, then $T^{-1}(I)$ is an interval and we have
\[
R\text{ is an $I$-itinerary }\ \Leftrightarrow \ XRX^{-1} \text{ is a $T^{-1}(I)$-itinerary}.
\]
\end{remark}

We will use another fact about itineraries of an interval exchange. Without loss of generality,
we consider $J=[0,1)$.
The intervals $J_X$ are left-closed right-open for all $X\in\{0,1\dots,k-1\}$.
Such interval exchange $T$ is right-continuous.
Therefore, if $I=[\gamma,\delta)$, then every word $w\in{\rm It}_I=\{R(x) \colon x\in I\}$ is an
$I$-itinerary $R(x)$ for infinitely many $x\in I$, which form an interval, again left-closed right-open.

\begin{proposition}\label{p:spojitost}
Let $T$ be a $k$-interval exchange satisfying the minimality condition and let $I=[\gamma,\delta)\subset[0,1)$.
There exist neighbourhoods $H_\gamma$ and $H_\delta$ of $\gamma$ and $\delta$, respectively, such that
 for every $\tilde{\gamma}\in H_\gamma$ and $\tilde{\delta}\in H_\delta$ with $0\leq \tilde{\gamma}<\tilde{\delta}\leq 1$ one has
\[
\mathit{It}_{\tilde{I}} \supseteq \mathit{It}_{{I}}\,,\quad \text{where } \tilde{I}=[\tilde{\gamma},\tilde{\delta}).
\]
In particular, if $\#\mathit{It}_{I}=k+2$, then $\mathit{It}_{\tilde{I}} = \mathit{It}_{{I}}$.
\end{proposition}

\pfz
Let $\mathit{It}_I=\{R_1,\dots,R_m\}$, $m\in\N$, and $I_i = \{ x \in I \colon R_I(x) = R_i\}$ for $1\leq i\leq m$. As already mentioned,
$m\leq k+2$.
For every $i=1,\dots,m$, consider arbitrary $x_i$ in the interior of $I_i$.
Such $x_i$ satisfies that $T^j(x_i)\notin\{\gamma,\delta\}$  for $0 \leq j \leq r_I(x_i)=|R_I(x_i)|$.
The reason is simple:  if $T^j(x_i)$ were equal to $\gamma$ (or $\delta$), then a point  $x$  in $I_i$    with $x< x_i$  (or $x>  x_i$) would have a longer  return time than $x_i$ itself, which is a contradiction.
Denote $M = \{T^j(x_i) \colon 0\leq j\leq r_I(x_i),\ i=1,\dots,m\}$ and
\[
\varepsilon:=\min\big\{|y-z| \colon y\in M,z\in \{\gamma,\delta\}\big\}\,.
\]
Let $H_\gamma=(\gamma-\varepsilon,\gamma+\varepsilon)$ and $H_\delta=(\delta-\varepsilon,\delta+\varepsilon)$ be
 neighbourhoods of $\gamma$ and $\delta$, respectively.
If $\tilde{\gamma}\in H_\gamma$ and $\tilde{\delta}\in H_\delta$ with $0\leq \tilde{\gamma}<\tilde{\delta}\leq 1$
 and $\tilde{I}=[\tilde{\gamma},\tilde{\delta})$, then clearly
for every $i=1,\dots,m$ we have  $x_i \in {\tilde{I}}$ and
\[
T^j(x_i)\in I \quad\Leftrightarrow\quad T^j(x_i)\in \tilde{I} \qquad \text{ for } 0\leq j\leq r_I(x_i)\,.
\]
Therefore the point $x_i$ has the same return time with respect  to $I$ as to
  $\tilde{I}$.  Consequently, the $\tilde{I}$-itinerary of $x_i$ coincides with the $I$-itinerary of $x_i$.
Thus $\mathit{It}_{{I}} \subseteq \mathit{It}_{\tilde{I}}$.
\pfk

For the rest of the section, we consider only symmetric interval exchange.   In order to state a property of such interval exchanges, for an interval
$K=[c,d)\subset [0,1)$ we set $\overline{K}=[1-d,1-c)$.  In this notation
\begin{equation}\label{zaTabuli}
T(J_X)=\overline{J_X}  \ \text{ for any letter $X\in\{0,1,\dots,k-1\}$\,.}
\end{equation}

\begin{proposition}\label{p:symetrie}
Let  $T:[0,1)\to[0,1)$ be a symmetric exchange of $k$ intervals
satisfying the minimality condition. Let $I\subset[0,1)$ and let
$R_1,\dots,R_m$ be the $I$-itineraries. The
$\overline{I}$-itineraries are the mirror images of the
$I$-itineraries, namely $\overline{R_1},\dots,\overline{R_m}$.
Moreover, if
\[
[\gamma_j,\delta_j):=\{x\in I \colon R_I(x)=R_j\}\quad\text{ and  }\quad [\gamma'_j,\delta'_j):=T_I[\gamma_j,\delta_j)\,,
\]
for $j=1,\dots,m$, then
\[
\{x\in \overline{I} \colon R_{\overline{I}}(x)=\overline{R_j}\} = [1-\delta'_j,1-\gamma'_j)\,.
\]
\end{proposition}

\pfz
Consider the restriction of the transformation $T$ to the set
\[
S=[0,1)\setminus\{T^j(\alpha)\colon j\in\Z,\ \alpha\text{ is a discontinuity of }T\}\,.
\]
Such a restriction is a bijection $S\to S$.
We will show by induction that for any $i\geq 1$ and $y\in S$
\begin{equation}\label{eq:ramec1}
T^{-i}(y) = 1-T^{i}(1-y)\,.
\end{equation}

Let $y \in S$ and $j \in \{0,\ldots, k-1\}$ such that $y \in I_j$.
Since $T$ is symmetric, we have
\[
1 - y \in I_j \ \Leftrightarrow \  y \in T(I_j).
\]
The last equivalence and the definition of $T$ imply
\begin{align*}
T(1 - y) &= 1 - y+c_j \qquad \qquad \text{ and } \\
T^{-1}(y) &= y - c_j.
\end{align*}
Summing the last two equalities we obtain
\[
T^{-1}(y) = 1-T(1-y).
\]
Then, using the induction hypothesis, we have for $y\in S$ that
\[
T^{-(i+1)}(y) = T^{-1}\big(T^{-i}(y)\big)=1-T\big(1-T^{-i}(y)\big) = 1- T\big(T^i(1-y)\big) = 1-T^{i+1}(1-y)\,,
\]
which proves~\eqref{eq:ramec1}.

Using \eqref{zaTabuli} we can write for $y\in S$
\begin{equation}\label{eq:ramec3}
T^{-1}(y)\in J_X \ \Leftrightarrow \ y\in T(J_X) \ \Leftrightarrow \ 1-y\in J_X\,.
\end{equation}
More generally,
\begin{equation}\label{eq:ramec2}
T^{-i}(y) = T^{-1}\big(T^{-(i-1)}(y)\big)\in J_X \ \Leftrightarrow \ 1-T^{-(i-1)}(y)=T^{i-1}(1-y)\in J_X\,,
\end{equation}
where we have first used~\eqref{eq:ramec3} and then~\eqref{eq:ramec1}.

Now we  show that if $R_j$ is an $I$-itinerary, then its mirror image $\overline{R_j}$ is an $\overline{I}$-itinerary.
Consider $\rho\in (\gamma_j,\delta_j) \cap S$ and let $R_I(\rho)=a_0a_1\cdots a_{n-1}$ be
its $I$-itinerary, i.e., $a_i=X$ if and only if $T^i(\rho)\in J_X$. Moreover, $T^i(\rho)\notin I$ for $1\leq i<n$, \ and $T^n(\rho)\in I$. Let
\begin{equation}\label{eq:hruzasymetrie}
\rho':=1-T^n(\rho) = 1-T_I(\rho) \in (1-\delta'_j,1-\gamma'_j)\cap S\subset \overline{I}\,.
\end{equation}
By~\eqref{eq:ramec1}, we have $\rho'=T^{-n}(1-\rho)$, and therefore again by~\eqref{eq:ramec1},
$T^i(\rho')=T^{-(n-i)}(1-\rho)=1-T^{n-i}(\rho)\notin \overline{I}$ for $0<i<n$. On the other hand, $T^n({\rho'})=1-\rho\in \overline{I}$.
By~\eqref{eq:ramec2}, we have for $i=0,1,\dots,n-1$ that
\[
J_X\ni T^i({\rho'}) = T^{-(n-i)}(1-\rho)\ \Leftrightarrow \ T^{n-i-1}(\rho)\in J_X\,,
\]
which implies that the $\overline{I}$-itinerary of $\rho'$ is $R_{\overline{I}}({\rho'})=a_{n-1}a_{n-2}\cdots a_0$, as we wanted to show.

By right continuity of $T$, all points from $[1-\delta'_j,1-\gamma'_j)$ have the same $\overline{I}$-itinerary as $\rho'\in(1-\delta'_j,1-\gamma'_j)\cap S$.
\pfk

The above auxiliary statements will be used in Section~\ref{sec:returntimes} for the description of $I$-itineraries in exchanges of three intervals.
Analogous result for exchange of two intervals was given in~\cite{MaPeItineraries}.
The claim of the last proposition also partially follows from the work done in \cite{FeHoZa}.

%%%%%%%%%%%%%%%%%%%%%%%%%%%%%%%%%%%%%%%%%%%%%%%%%%%%%%%%%%%%%%%%%%%%%%%%%%
\section{Exchange of three intervals}\label{sec:3iet}

We will be particularly interested in exchange of three intervals. For reasons that will
appear later, we prefer to use for its coding the ternary alphabet $\{A,B,C\}$
instead of $\{0,1,2\}$. Without loss of generality let $0<\alpha<\beta<1$. Let $T:[0,1)\to[0,1)$ be given by
\begin{equation}\label{eq:3IET}
T(x)=\begin{cases}
x+1-\alpha & \text{if } x\in[0,\alpha)=: J_A\,,\\
x+1-\alpha-\beta & \text{if } x\in[\alpha,\beta)=: J_B\,,\\
x-\beta & \text{if } x\in[\beta,1)=: J_C\,.
\end{cases}
\end{equation}
The transformation $T$ is an exchange of three intervals with the permutation (321). It is often called a \textit{3iet} for short.
The infinite word ${\bf u}_\rho$ coding the orbit of a point $\rho\in[0,1)$ under a 3iet is called a \textit{3iet word}.

We require that $1-\alpha$ and $\beta$ be linearly independent over $\Q$, which is known to be a necessary and sufficient condition for minimality of
the 3iet $T$.
Non-degeneracy of $T$ is equivalent to the condition of minimality together with
\begin{equation}\label{eq:nondeg}
1\notin (1-\alpha)\Z + \beta\Z\,,
\end{equation}
see \cite{FeHoZa}.
This means that the 3iet word ${\bf u}$ has complexity ${\mathcal C}_{\bf u}(n)=2n+1$
if and only if the parameters $\alpha$ and $\beta$ of the corresponding 3iet $T$ satisfy~\eqref{eq:nondeg}.

For a given subinterval $I\subset [0,1)$ there exist at most five $I$-itineraries under a 3iet $T$.
In particular, from the paper of Keane~\cite{Keane}, one can deduce what are the intervals of
 points with the same itinerary. We summarize it as the following lemma.

\begin{lemma}\label{l:keane}
Let $T$ be a 3iet defined by~\eqref{eq:3IET} and let $I=[\gamma,\delta)\subset[0,1)$ such that $\delta<1$. Denote
\[
\begin{aligned}
k_\alpha &:=\min\big\{k\in\Z,\, k\geq 0 \colon T^{-k}(\alpha)\in (\gamma,\delta)\big\}\,,\\
k_\beta &:=\min\big\{k\in\Z,\, k\geq 0 \colon T^{-k}(\beta)\in(\gamma,\delta) \big\}\,,\\
k_\gamma &:=\min\big\{k\in\Z,\, k\geq 1 \colon T^{-k}(\gamma)\in(\gamma,\delta)\big\}\,,\\
k_\delta &:=\min\big\{k\in\Z,\, k\geq 1 \colon T^{-k}(\delta)\in(\gamma,\delta)\big\}\,,
\end{aligned}
\]
and further
\[
\a:= T^{-k_\alpha}(\alpha),\ \b:= T^{-k_\beta}(\beta),\ \c:= T^{-k_\gamma}(\gamma),\ \d:= T^{-k_\delta}(\delta).
\]
For $x\in I$, let $K_x$ be a maximal interval such that for every $y\in K_x$, we have $R(y)=R(x)$. Then $K_x$ is of the form $[c,d)$ with
$c,d \in \{ \gamma,\delta,\a,\b,\c,\d \}$. Consequently, $\#{\rm It}_I\leq 5$.
\end{lemma}

For a 3iet $T$, Lemma~\ref{l:keane} implies the already mentioned result that $T_I$ is an exchange of at most 5 intervals. Consequently,
the transformation $T_I$ has at most four discontinuity points. In fact, the following result of~\cite{BaMaPe}
says that independently of the number of $I$-itineraries,
the induced map $T_I$ has always at most two discontinuity points.

\begin{proposition}[\cite{BaMaPe}]\label{p:bamape3iet}
Let $T:J\to J$ be a 3iet with the permutation $(321)$ and satisfying the minimality condition, and let $I\subset J$
be an interval. The first return map $T_I$ is either a 3iet with permutation $(321)$ or a 2iet with permutation $(21)$.
In particular, in the notation of Lemma~\ref{l:keane}, we have $\d\leq \c$.
\end{proposition}

\noindent{\bf Convention:} For the rest of the paper, let $T$ be a non-degenerate exchange of three intervals with permutation (321) given by~\eqref{eq:3IET}.

%%%%%%%%%%%%%%%%%%%%%%%%%%%%%%%%%%%%%%%%%%%%%%%%%%%%%%%%%%%%%%%%%%%%%%%%%%
\section{Return time in a 3iet}\label{sec:returntimes}

The aim of this section is to describe the possible return times of a non-degenerate $3$iet $T$ to a general subinterval $I\subset[0,1)$.
Our aim is to prove the following theorem.

\begin{theorem}\label{t:returntime}
Let $T$ be a non-degenerate 3iet and let $I\subset[0,1)$. There exist positive integers $r_1,r_2$ such that the
return time of any $x\in I$ takes value in the set
$\{r_1,r_1+1,r_2,r_1+r_2,r_1+r_2+1\}$ or $\{r_1,r_1+1,r_2,r_2+1,r_1+r_2+1\}$.
\end{theorem}

First, we will formulate an important lemma, which needs the following notation. Given letters $X,Y,Z\in\{A,B,C\}$ and a finite word $w\in\{A,B,C\}^*$,
Let $\omega_{XY\to Z}(w)$ be the set of words obtained from $w$ replacing one factor $XY$ by the letter $Z$, i.e.
\[
\omega_{XY\to Z}(w) = \{w_1Zw_2 \colon w=w_1XYw_2\}\,.
\]
Similarly,
\[
\omega_{Z\to XY}(w) = \{w_1XYw_2 \colon w=w_1Zw_2\}\,.
\]
Clearly,
\begin{equation}\label{eq:iff}
v\in\omega_{XY\to Z}(w) \ \Leftrightarrow\ w\in\omega_{Z\to XY}(v).
\end{equation}
By abuse of notation, we write $v=\omega_{XY\to Z}(w)$ instead of $v\in \omega_{XY\to Z}(w)$.

\begin{lemma}\label{l:jakitiner}
Assume that the orbits of points $\alpha, \beta, \gamma$ and $\delta$ are mutually disjoint.
For sufficiently small $\varepsilon>0$, we have the following relations between $I$-itineraries of points in $I$
%{\renewcommand{\labelenumi}{(\alph{enumi})}
\begin{enumerate}[(a)]
\item \label{l:jakitiner_1} $R(\a-\varepsilon) = \omega_{B\to AC}\big(R(\a+\varepsilon)\big)$, \label{enum:A-}
\item \label{l:jakitiner_2} $R(\a+\varepsilon) = \omega_{AC\to B}\big(R(\a-\varepsilon)\big)$, \label{enum:A+}
\item \label{l:jakitiner_3} $R(\b-\varepsilon) = \omega_{CA\to B}\big(R(\b+\varepsilon)\big)$, \label{enum:B-}
\item \label{l:jakitiner_4} $R(\b+\varepsilon) = \omega_{B\to CA}\big(R(\b-\varepsilon)\big)$, \label{enum:B+}
\item \label{l:jakitiner_5} $R(\d+\varepsilon) = R(\d-\varepsilon)R(\delta-\varepsilon)$, \label{enum:D+}
\item \label{l:jakitiner_6} $R(\c-\varepsilon) = R(\c+\varepsilon)R(\gamma+\varepsilon)$, \label{enum:C-}
\end{enumerate}
where $\a,\b,\c,\d$ are given in Lemma~\ref{l:keane}.
%}
\end{lemma}

\pfz
We will first demonstrate the proof of the case \ref{l:jakitiner_1}.
Let $K = [\a-\varepsilon, \a+\varepsilon]$ with $\varepsilon$ chosen such that $K\subset I$ and
\begin{equation}\label{l:j_d2}
\alpha,\beta,\gamma,\delta \not \in T^i(K) \text{ for all }0 \leq i \leq k_\alpha
\text{ with the only exception of } T^{k_\alpha}(\a) = \alpha.
\end{equation}
For simplicity, denote $t=\max\{r_I(x) \colon x\in K\}$ the maximal return time.
The existence of such $\varepsilon$ follows trivially from the definition of the interval exchange transformation and the assumptions of the lemma.

Let $K_- = [\a - \varepsilon,\a)$ and $K_+ = [\a,\a+\varepsilon]$.
It follows from the definition of $\a$ and condition \eqref{l:j_d2} that for all $i$ such that $0 < i \leq k_\alpha$ we have $T^i(K) \cap I = \emptyset$.
Moreover, condition \eqref{l:j_d2} implies that all such $T^i(K)$ are intervals.
It implies that for any $x,y \in K$, the prefixes of $R(x)$ and $R(y)$ of length $k_\alpha+1$ are the same.
Denote this prefix by $w$.

The definition of $k_\alpha$ implies that $\alpha \in T^{k_\alpha}(K)$.
Since $T^{k_\alpha}(K_+) = [\alpha,\alpha+\varepsilon] \subset J_B$, we obtain
\[
T^{k_\alpha+1}(K_+) = \big[T(\alpha),T(\alpha)+\varepsilon\big)\,.
\]
Furthermore, since $T^{k_\alpha}(K_-)  = [\alpha-\varepsilon,\alpha) \subset J_A$, we obtain
\[
T^{k_\alpha+1}(K_-) =  [1-\varepsilon,1)\subset J_C\,,
\]
and thus
\[
T^{k_\alpha+2}(K_-) =  \big[T(\alpha)-\varepsilon,T(\alpha)\big)\,.
\]
This implies that the set $K'=T^{k_\alpha+2}(K_-) \cup T^{k_\alpha+1}(K_+) = [T(\alpha)-\varepsilon,T(\alpha)+\varepsilon]$ is an interval.
As above, condition \eqref{l:j_d2} implies that the set $T^i(K')$ is an interval for all $i$ such that $0 \leq i \leq t - k_\alpha - 1$.
It follows that $\min \{ i \colon T^i(K') \cap K \neq \emptyset \}  = t - k_\alpha - 2$ and condition \eqref{l:j_d2} moreover implies that $T^{t-k_\alpha-2}(K') \subset K$.
Thus, the iterations $x,T(x),\dots,T^{t-k_\alpha-2}(x)$ of every $x \in K'$ are coded be the same word, say $v$.

%%%%%%%%%%%%%%%%%%%%%%%%%%%%%%%%%%%%%%%%%%%%%%%%%%%%
%%%%%%%%%%%%%%%%%%%%%%%%%%%%%%%%%%%%%%%%%%%%%%%%%%%%
%%%%%%%%%%%%%%%%%%%%%%%%%%%%%%%%%%%%%%%%%%%%%%%%%%%%
% spolecne definice pro obrazky
%%%%%%%%%%%%%%%%%%%%%%%%%%%%%%%%%%%%%%%%%%%%%%%%%%%%
%%%%%%%%%%%%%%%%%%%%%%%%%%%%%%%%%%%%%%%%%%%%%%%%%%%%
%%%%%%%%%%%%%%%%%%%%%%%%%%%%%%%%%%%%%%%%%%%%%%%%%%%%
%[styl popisovaciho uzlu], trida cary, x-souradnice, popis, nazev bodu
\newcommand\ietpoint[5][pos=1,below]{
    \draw[style=#2] (#3,0.12) -- (#3,-0.12) node[#1]{#4};
    \node[ietinvisible] (#5) at (#3,0){};
}
\tikzset{ietinvisible/.style={outer sep=4,inner sep=0,minimum size=0}}
\newcommand\ietbasic[6][]{
\begin{scope}[every path/.style={very thick},#1]
  \node(start)[at={(0,0)},label=below:{$0$}]{};
  \node(end)[at={(100,0)},label=below:{$1$}]{};
  \draw[[-)] (start.center) -- (end.center);
  \ietpoint{discoAl}{#2}{$\alpha$}{alpha}
  \ietpoint{discoBe}{#3}{$\beta$}{beta}
  %indukovany
  \node(startI)[at={(#4,0)},label=below:{$\gamma$}]{};
  \node(endI)[at={(#5,0)},label=below:{$\delta$}]{};
  \draw[[-),thick] (startI.center) -- (endI.center);
   \fill[opacity = 0.1, black,rounded corners=5] (#4,-0.2) -- (#5, -0.2) -- (#5, 0.2) -- (#4, 0.2) -- cycle;
  #6
\end{scope}
}
\newcounter{ietstep}
\newcounter{ietstepP}
\newcounter{ietstepPP}
\newcommand\ietstepcounters{   \stepcounter{ietstep}    \stepcounter{ietstepP}    \stepcounter{ietstepPP}}

%%%%%%%%%%%%%%%%%%%%%%%%%%%%%%%%%%%%%%%%%%%%%%%%%%%%
%%%%%%%%%%%%%%%%%%%%%%%%%%%%%%%%%%%%%%%%%%%%%%%%%%%%
%%%%%%%%%%%%%%%%%%%%%%%%%%%%%%%%%%%%%%%%%%%%%%%%%%%%
% konec spolecnych definic
%%%%%%%%%%%%%%%%%%%%%%%%%%%%%%%%%%%%%%%%%%%%%%%%%%%%
%%%%%%%%%%%%%%%%%%%%%%%%%%%%%%%%%%%%%%%%%%%%%%%%%%%%
%%%%%%%%%%%%%%%%%%%%%%%%%%%%%%%%%%%%%%%%%%%%%%%%%%%%

\begin{figure}[ht]
\centering
\begin{tikzpicture}[]
\newcommand*\ietalpha{44} %kde mame 1. diskontinuitu (0 az 100)
\newcommand*\ietbeta{90} %kde mame 2. diskontinuitu (0 az 100)
\newcommand*\ietgamma{17} %kde mame 2. diskontinuitu (0 az 100)
\newcommand*\ietdelta{30} %kde mame 2. diskontinuitu (0 az 100)
\newcommand*\ietA{25}
\newcommand*\ietsdist{3}
\newcommand*\ietpad{0.2}
\newcommand*\ietepsstep{1.5}

\setcounter{ietstep}{1}
\setcounter{ietstepP}{0}
\setcounter{ietstepPP}{-1}
\tikzset{disco/.style={color=gray}}
\tikzset{discoAl/.style={style=disco}}
\tikzset{discoBe/.style={style=disco}}
\tikzset{hlpo/.style={very thick}}
\tikzset{hlpoA/.style={yscale=1.2}}
\tikzset{tmap/.style={ultra thick,}}
\tikzset{tmapG/.style={thick,dashed,color=gray}}
 \begin{scope}[font=\footnotesize,scale=0.75,x=0.008\textwidth,inner xsep=0,inner ysep=3.5]
 \ietbasic[yshift={(1-\number\value{ietstep})*\ietsdist cm}]{\ietalpha}{\ietbeta}{\ietgamma}{\ietdelta}{
  \ietpoint[pos=0.3,above]{hlpoA}{\ietA}{$\a$}{pointA\arabic{ietstep}}
  \ietpoint[at={(\ietA-\ietepsstep-4,0.5)},anchor=south,inner ysep=0,inner xsep=0]{hlpo}{\ietA-\ietepsstep}{$\a-\varepsilon$}{pointAm\arabic{ietstep}}
   \draw[dotted,color=gray] (\ietA-\ietepsstep,0) -- (\ietA-\ietepsstep-4,0.5);
  \ietpoint[at={(\ietA+\ietepsstep+4,0.5)},anchor=south,inner ysep=0,inner xsep=0]{hlpo}{\ietA+\ietepsstep}{$\a+\varepsilon$}{pointAp\arabic{ietstep}}
     \draw[dotted,color=gray] (\ietA+\ietepsstep,0) -- (\ietA+\ietepsstep+4,0.5);
 }

%%%%%%%%%%%%%%% PATRO 2
  \ietstepcounters
  \tikzset{discoAl/.style={color=white}}
 \ietbasic[yshift={(1-\number\value{ietstep})*\ietsdist cm}]{\ietalpha}{\ietbeta}{\ietgamma}{\ietdelta}{
  \ietpoint[at={(\ietalpha+15,1)},inner ysep=0,inner xsep=0]{hlpoA}{\ietalpha}{$T^{k_\alpha}(\a) = \alpha$}{pointA\arabic{ietstep}}
  \draw[dotted,color=gray] (\ietalpha,0) -- (\ietalpha+15,1);
  \ietpoint[at={(\ietalpha-\ietepsstep,0)},anchor=south east]{hlpo}{\ietalpha-\ietepsstep}{$T^{k_\alpha}(\a-\varepsilon)$}{pointAm\arabic{ietstep}}
  \ietpoint[at={(\ietalpha+\ietepsstep,0)},anchor=north west]{hlpo}{\ietalpha+\ietepsstep}{$T^{k_\alpha}(\a+\varepsilon)$}{pointAp\arabic{ietstep}}
 }
 \draw[->,tmap] (pointAm\arabic{ietstepP}) to[out=-100,in=80]  node[pos=0.55,left,inner xsep=5] {$T^{k_\alpha}$} (pointAm\arabic{ietstep});
 \draw[->,tmap] (pointAp\arabic{ietstepP}) to[out=-100,in=80]  node[pos=0.4,right,inner xsep=10] {$T^{k_\alpha}$} (pointAp\arabic{ietstep});

%%%%%%%%%%%%%%% PATRO 3
   \ietstepcounters
     \tikzset{discoAl/.style={style=disco}}
    \ietbasic[yshift={(1-\number\value{ietstep})*\ietsdist cm},xshift=185]{\ietalpha}{\ietbeta}{\ietgamma}{\ietdelta}{
      \ietpoint[at={(100-\ietepsstep,0)},anchor=south east]{hlpo}{100-\ietepsstep}{$T^{k_\alpha+1}(\a-\varepsilon)$}{pointAm\arabic{ietstep}}
    }
    \draw[->,tmap] (pointAm\arabic{ietstepP}) to[out=-90,in=190] ($(pointAm\arabic{ietstepP})!0.5!(pointAm\arabic{ietstep})$) to[out=10,in=70]  node[pos=0.1,above] {$T$} (pointAm\arabic{ietstep});

%%%%%%%%%%%%%%% PATRO 4
   \ietstepcounters
  \ietbasic[yshift={(1-\number\value{ietstep})*\ietsdist cm},xshift=0]{\ietalpha}{\ietbeta}{\ietgamma}{\ietdelta}{
      \ietpoint[at={(100-\ietbeta-\ietepsstep,0)},anchor=south east]{hlpo}{100-\ietbeta-\ietepsstep}{$T^{k_\alpha+2}(\a-\varepsilon)$}{pointAm\arabic{ietstep}}
        \ietpoint[at={(100-\ietbeta+\ietepsstep,0)},anchor=south west]{hlpo}{100-\ietbeta+\ietepsstep}{$T^{k_\alpha+1}(\a+\varepsilon)$}{pointAp\arabic{ietstep}}
    }
            \draw[->,tmap] (pointAm\arabic{ietstepP}) to[out=-100,in=-10] ($(pointAm\arabic{ietstepP})!0.55!(pointAm\arabic{ietstep})$) to[out=170,in=90]  node[pos=0.1,above] {$T$} (pointAm\arabic{ietstep});
            \draw[->,tmap] (pointAp\arabic{ietstepPP}) to[out=-90,in=30] ($(pointAp\arabic{ietstepPP})!0.5!(pointAp\arabic{ietstep})$) to[out=-150,in=90]  node[pos=0.1,above] {$T$} (pointAp\arabic{ietstep});

%%%%%%%%%%%%%%% PATRO 5
   \ietstepcounters
  \ietbasic[yshift={(1-\number\value{ietstep})*\ietsdist cm},xshift=0]{\ietalpha}{\ietbeta}{\ietgamma}{\ietdelta}{
      \ietpoint[at={(\ietgamma+\ietdelta-\ietA-\ietepsstep,0)},anchor=south east]{hlpo}{\ietgamma+\ietdelta-\ietA-\ietepsstep}{$T^{t}(\a-\varepsilon)$}{pointAm\arabic{ietstep}}
      \ietpoint[at={(\ietgamma+\ietdelta-\ietA+\ietepsstep,0)},anchor=south west]{hlpo}{\ietgamma+\ietdelta-\ietA+\ietepsstep}{$T^{t-1}(\a+\varepsilon)$}{pointAp\arabic{ietstep}}
    }
 \draw[->,tmap] (pointAm\arabic{ietstepP}) to[out=-90,in=90]  node[pos=0.55,left,inner xsep=5] {$T^{t-k_\alpha-2}$} (pointAm\arabic{ietstep});
 \draw[->,tmap] (pointAp\arabic{ietstepP}) to[out=-90,in=90]  node[pos=0.45,right,inner xsep=5] {$T^{t-k_\alpha-2}$} (pointAp\arabic{ietstep});
\end{scope}
\end{tikzpicture}
\caption{Situation in the proof of Lemma~\ref{l:jakitiner}, case \ref{l:jakitiner_1}.}
\label{fig:jakitiner_1}.
\end{figure}
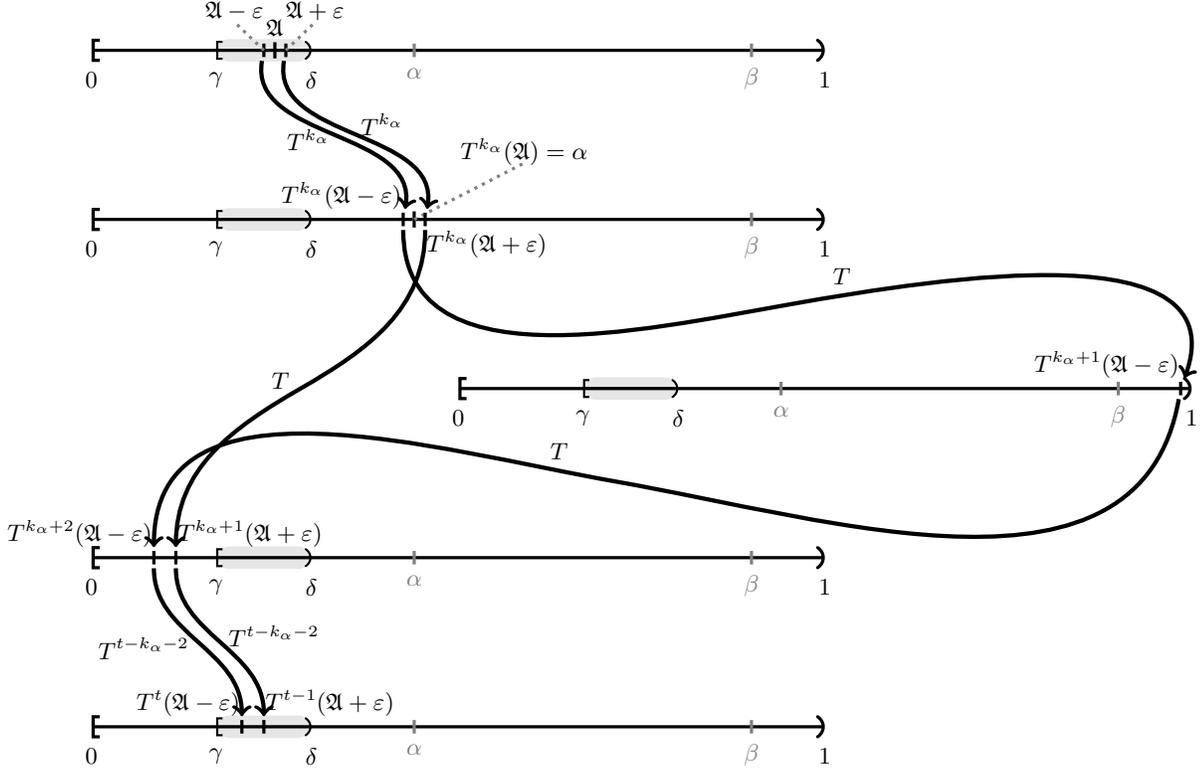

The whole situation is depicted in Figure~\ref{fig:jakitiner_1}.
From what is said above, we can write down the $I$-itineraries of points from $K$,
\[
R(x) = \begin{cases} wACv  & \text { if } x \in K_- ,\\
 wBv  & \text{ if } x \in K_+.\end{cases}
\]
This finishes the proof of \eqref{l:jakitiner_1}.

The claim in item \eqref{l:jakitiner_3} is analogous to \eqref{l:jakitiner_1}. Cases \eqref{l:jakitiner_2} and \eqref{l:jakitiner_4}
are derived from \eqref{l:jakitiner_1} and \eqref{l:jakitiner_3} by the use of equivalence~\eqref{eq:iff}.

\vspace{\baselineskip}
Let us now demonstrate the proof of the case \eqref{l:jakitiner_5}. Denote $s = \min\{n\in\Z^+ \colon T^n(\delta)\in I\}$.
Let $K = [\d-\varepsilon, \d+\varepsilon]$ with $\varepsilon$ chosen such that $K \subset I$ and
\begin{equation}\label{l:j_d1}
\alpha,\beta,\gamma,\delta \not \in T^i(K) \text{ for all }0 \leq i \leq k_\delta+s \text{ with the only exception of }T^{k_\delta}(\d) = \delta\,.
\end{equation}
The existence of such $\varepsilon$ follows trivially from the definition of the interval exchange transformation and the assumptions of the lemma.

Condition \eqref{l:j_d1} implies that $T^i(K)$ is an interval for all $i$ such that $0 < i \leq k_\delta+s$.
Moreover, $T^i(K) \cap I = \emptyset$ for all $i$ such that $0 < i <k_\delta$.
We obtain $T^{k_\delta}(K) \cap I = [\delta-\varepsilon,\delta)$.
In other words, the $I$-itineraries of all points of $K$ start with a prefix of length $k_\delta$ which is equal to $R(\d-\varepsilon)$.
Condition \eqref{l:j_d1} and the definition of $s$ implies that for all $i$ such that $k_\delta < i < s+k_\delta$ we have
$T^i(K)\subset J_X$ for some $X\in\{A,B,C\}$ and $T^i(K) \cap I = \emptyset$. Moreover, $T^{i}(K) \subset I$ for $i=k_\delta+s$.
Thus, the iterations of points of $T^{k_\delta}(K)=[\delta-\varepsilon,\delta)\cup T^{k_\delta}[\d,\d+\varepsilon]$
are coded by the same word of length $s$, namely $R(\delta-\varepsilon)$. Altogether,  we can conclude that the $I$-itinerary of points
in the interval $[\d, \d+\varepsilon]$ is equal to $R(\d-\varepsilon)R(\delta-\varepsilon)$.
The situation is depicted in Figure~\ref{fig:jakitiner_4}.

\begin{figure}[ht]
\centering
\begin{tikzpicture}[]
\newcommand*\ietalpha{25} %kde mame 1. diskontinuitu (0 az 100)
\newcommand*\ietbeta{85} %kde mame 2. diskontinuitu (0 az 100)
\newcommand*\ietgamma{33} %kde mame 2. diskontinuitu (0 az 100)
\newcommand*\ietdelta{52} %kde mame 2. diskontinuitu (0 az 100)
\newcommand*\ietA{41} %souradnice bodu D !
\newcommand*\ietsdist{3}
\newcommand*\ietpad{0.2}
\newcommand*\ietepsstep{1.5}

\setcounter{ietstep}{1}
\setcounter{ietstepP}{0}
\setcounter{ietstepPP}{-1}

\tikzset{disco/.style={color=gray}}
\tikzset{discoAl/.style={style=disco}}
\tikzset{discoBe/.style={style=disco}}
\tikzset{hlpo/.style={very thick}}
\tikzset{hlpoA/.style={yscale=1.2}}
\tikzset{tmap/.style={ultra thick,}}
\tikzset{tmapG/.style={thick,dashed,color=gray}}

 \begin{scope}[font=\footnotesize,scale=0.75,x=0.008\textwidth,inner xsep=0,inner ysep=3.5]

 \ietbasic[yshift={(1-\number\value{ietstep})*\ietsdist cm}]{\ietalpha}{\ietbeta}{\ietgamma}{\ietdelta}{

  \ietpoint[pos=0.3,above]{hlpoA}{\ietA}{$\d$}{pointA\arabic{ietstep}}
  \ietpoint[at={(\ietA-\ietepsstep-4,0.5)},anchor=south,inner ysep=0,inner xsep=0]{hlpo}{\ietA-\ietepsstep}{$\d-\varepsilon$}{pointAm\arabic{ietstep}}
    \draw[dotted,color=gray] (\ietA-\ietepsstep,0) -- (\ietA-\ietepsstep-4,0.5);
  \ietpoint[at={(\ietA+\ietepsstep+4,0.5)},anchor=south,inner ysep=0,inner xsep=0]{hlpo}{\ietA+\ietepsstep}{$\d+\varepsilon$}{pointAp\arabic{ietstep}}
  \draw[dotted,color=gray] (\ietA+\ietepsstep,0) -- (\ietA+\ietepsstep+4,0.5);
 }

%%%%%%%%%%%%%%% PATRO 2
  \ietstepcounters

 \ietbasic[yshift={(1-\number\value{ietstep})*\ietsdist cm}]{\ietalpha}{\ietbeta}{\ietgamma}{\ietdelta}{

  \ietpoint[at={(\ietdelta-\ietepsstep,0)},anchor=south east]{hlpo}{\ietdelta-\ietepsstep}{$T^{k_\delta}(\d-\varepsilon)$}{pointAm\arabic{ietstep}}
  \ietpoint[at={(\ietdelta+\ietepsstep,0)},anchor=north west]{hlpo}{\ietdelta+\ietepsstep}{$T^{k_\delta}(\d+\varepsilon)$}{pointAp\arabic{ietstep}}
 }

 \draw[->,tmap] (pointAm\arabic{ietstepP}) to[out=-100,in=80]  node[pos=0.55,left,inner xsep=5] {$T^{k_\delta}$} (pointAm\arabic{ietstep});
 \draw[->,tmap] (pointAp\arabic{ietstepP}) to[out=-100,in=80]  node[pos=0.4,right,inner xsep=10] {$T^{k_\delta}$} (pointAp\arabic{ietstep});

%%%%%%%%%%%%%%% PATRO 3

   \ietstepcounters

  \ietbasic[yshift={(1-\number\value{ietstep})*\ietsdist cm},xshift=0]{\ietalpha}{\ietbeta}{\ietgamma}{\ietdelta}{
      \ietpoint[at={(\ietgamma+\ietdelta-\ietA-\ietepsstep,0)},anchor=south east]{hlpo}{\ietgamma+\ietdelta-\ietA-\ietepsstep}{$T^{t}(\delta-\varepsilon)$}{pointAm\arabic{ietstep}}
      \ietpoint[at={(\ietgamma+\ietdelta-\ietA+\ietepsstep,0)},anchor=south west]{hlpo}{\ietgamma+\ietdelta-\ietA+\ietepsstep}{$T^{t}(\delta+\varepsilon)$}{pointAp\arabic{ietstep}}
    }

 \draw[->,tmap] (pointAm\arabic{ietstepP}) to[out=-90,in=90]  node[pos=0.55,left,inner xsep=5] {$T^{t}$} (pointAm\arabic{ietstep});
 \draw[->,tmap] (pointAp\arabic{ietstepP}) to[out=-90,in=90]  node[pos=0.45,right,inner xsep=5] {$T^{t}$} (pointAp\arabic{ietstep});

%\draw[-,tmap,line width=2*\ietepsstep*0.008*\textwidth,color=gray,opacity=0.2,xshift=100] ($(pointAm\arabic{ietstepP})+(\ietepsstep,-4pt)$) to[out=-90,in=90]  ([ietinvisible]$(pointAm\arabic{ietstep})+(\ietepsstep,4pt)$);

%  \draw[-,tmap,line width=4*\ietepsstep,color=gray,opacity=0.5] ($(pointAm\arabic{ietstepP})!0.5!(pointAp\arabic{ietstepP})$) to[out=-90,in=90] ($(pointAm\arabic{ietstep})!0.5!(pointAp\arabic{ietstep})$);

\end{scope}
%  \labelnum{1}{All integers are labeled}
  %\labelnum[yshift=-2cm]{2}{Even numbers are labeled}
\end{tikzpicture}
\caption{Situation in the proof of Lemma~\ref{l:jakitiner}, case \ref{l:jakitiner_4}.}
\label{fig:jakitiner_4}.
\end{figure}
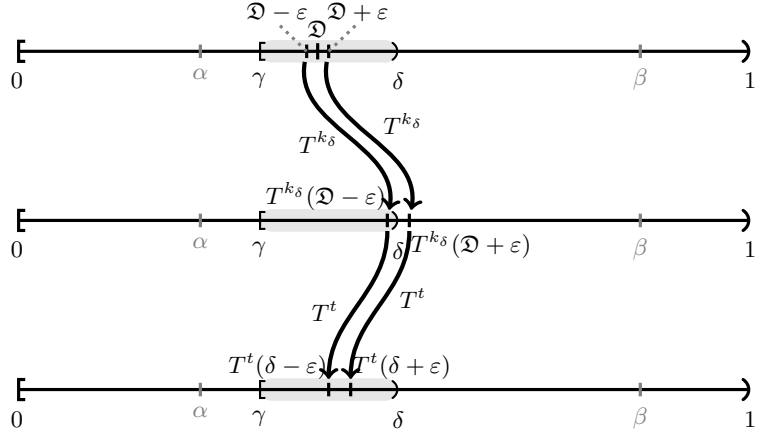

Case \eqref{l:jakitiner_6} can be treated in a way analogous to case \eqref{l:jakitiner_5}.
\pfk

Now we are in the state to prove the main theorem describing the return times in 3iet.
In the proof, it is sufficient to focus on the case when $\#{\rm It}_I=5$, since, as we have seen from
Proposition~\ref{p:spojitost}, the set of $I$-itineraries, and thus also their return times, for the other
 cases is only a subset of ${\rm It}_{\tilde{I}}$ for some ``close enough'' generic subinterval $\tilde{I}\subset[0,1)$.
So throughout the rest of this section, suppose that $\#{\rm It}_I=5$. This means by Lemma~\ref{l:keane} that points $\a, \b, \c, \d$ lie in the interior
of the interval $I=[\gamma, \delta)$ and are mutually distinct, moreover, by Proposition~\ref{p:bamape3iet}, we have $\d<\c$.
Such conditions imply 12 possible orderings of $\a,\b,\c,\d$ which give rise to 12 cases in the study of return times.
 We will describe them in the proof of Theorem~\ref{t:returntime}
as cases (i)--(xii) and then show in Example~\ref{ex:vsech12} that all 12 cases may occur.

\begin{remark} \label{re:sedi_na_kraji}
Note that if $\gamma=0$, i.e.\ we induce on an interval $I=[0,\delta)$, we have $T^{-1}(\gamma)=\beta$ and therefore necessarily $\b=\c$.
Thus there are at most four $I$-itineraries.
Due to Proposition~\ref{p:symetrie}, similar situation happens if $\delta=1$.
\end{remark}

\pfz[Proof of Theorem~\ref{t:returntime}]
We will discuss the 12 possibilities of ordering of points $\a,\b,\c,\d$ in the interior of the interval $[\gamma,\delta)$ with the condition $\d<\c$.
The structure of the set of $I$-itineraries will be best shown in terms of $I$-itineraries of points in the left neighbourhood of the point $\d$
and right neighbourhood of the point $\c$. For simplicity, we thus denote for sufficiently small positive~$\varepsilon$
\[
R_1=R(\d-\varepsilon),\ R_2=R(\c+\varepsilon)\quad\text{ and }\quad |R_1|=t_1, |R_2|=t_2.
\]

In order to be allowed to use Lemma~\ref{l:jakitiner}, we will assume that
the orbits of points $\alpha, \beta, \gamma$ and $\delta$ are mutually disjoint. Otherwise, we use
Proposition~\ref{p:spojitost} to find a modified interval $\tilde{I}$ where this is satisfied and ${\rm It}_{\tilde{I}}= {\rm It}_I$.
\begin{enumerate}[(i)]
\item Let $\a<\b<\d<\c$. We know that $R(x)$ is constant on the intervals $[\gamma,\a)$, $[\a,\b)$, $[\b,\d)$, $[\d,\c)$, and $[\c,\delta)$.
By definition $R(x)=R_2$ for $x\in [\c,\delta)$ and $R(x)=R_1$ for
$x\in[\b,\d)$. We can derive from rule (\ref{enum:D+}) of
Lemma~\ref{l:jakitiner} that if $x\in[\d,\c)$, then $R(x)=R_1R_2$.
Further, we use rule (\ref{enum:B-}) to show that
$R(x)=\omega_{CA\to B}(R_1)$ for $x\in[\a,\b)$ and further by
applying rule (\ref{enum:A-}), we obtain that $R(x)=\omega_{B\to
AC}\big(\omega_{CA\to B}(R_1)\big)$ for $x\in[\gamma,\a)$.
Summarized,
\[
R(x)=
\begin{cases}
\omega_{B\to AC}\big(\omega_{CA\to B}(R_1)\big) &\text{ for }x\in[\gamma,\a)\\
\omega_{CA\to B}(R_1) &\text{ for }x\in[\a,\b)\\
R_1&\text{ for }x\in[\b,\d)\\
R_1R_2 &\text{ for }x\in[\d,\c)\\
R_2 &\text{ for }x\in[\c,\delta).
\end{cases}
\]
It is easy to show that the lengths of the above $I$-itineraries are $t_1$, $t_1-1$, $t_1$, $t_1+t_2$, $t_2$, respectively.
Setting $r_1=t_1-1$ and $r_2=t_2$, we obtain the desired return times.

The proofs of the other cases are analogous, we state the results in terms of $R_1$ and $R_2$.

\item Let $\d<\c<\a<\b$. We obtain
\[
R(x)=
\begin{cases}
R_1  &\text{ for }x\in[\gamma,\d)\\
R_2R_1&\text{ for }x\in[\d,\c)\\
R_2&\text{ for }x\in[\c,\a)\\
\omega_{AC\to B}(R_2) &\text{ for }x\in[\a,\b)\\
\omega_{B\to CA}\big(\omega_{AC\to B}(R_2)\big) &\text{ for }x\in[\b,\delta),
\end{cases}
\]
with lengths $t_1$, $t_1+t_2$, $t_2$, $t_2-1$, $t_2$, respectively. We set $r_1=t_1$ and $r_2=t_2-1$.

\item Let $\b<\a<\d<\c$. A discussion as above leads to
\[
R(x)=
\begin{cases}
\omega_{CA\to B}\big(\omega_{B\to AC}(R_1)\big) &\text{ for }x\in[\gamma,\b)\\
\omega_{B\to AC}(R_1) &\text{ for }x\in[\b,\a)\\
R_1&\text{ for }x\in[\a,\d)\\
R_1R_2 &\text{ for }x\in[\d,\c)\\
R_2 &\text{ for }x\in[\c,\delta),
\end{cases}
\]
 with the corresponding lengths $t_1$, $t_1+1$, $t_1$, $t_1+t_2$, $t_2$, respectively. We set $r_1=t_1$ and $r_2=t_2$.

\item Let $\d<\c<\b<\a$. We obtain
\[
R(x)=
\begin{cases}
R_1  &\text{ for }x\in[\gamma,\d)\\
R_2R_1&\text{ for }x\in[\d,\c)\\
R_2&\text{ for }x\in[\c,\b)\\
\omega_{B\to CA}(R_2) &\text{ for }x\in[\b,\a)\\
\omega_{AC\to B}\big(\omega_{B\to CA}(R_2)\big) &\text{ for }x\in[\a,\delta),
\end{cases}
\]
with lengths $t_1$, $t_1+t_2$, $t_2$, $t_2+1$, $t_2$, respectively. We set $r_1=t_1$ and $r_2=t_2$.

\item Let $\a<\d<\b<\c$. We obtain
\[
R(x)=
\begin{cases}
\omega_{B\to AC}(R_1)  &\text{ for }x\in[\gamma,\a)\\
R_1&\text{ for }x\in[\a,\d)\\
R_1R_2&\text{ for }x\in[\d,\b)\\
R_2\omega_{B\to AC}(R_1) &\text{ for }x\in[\b,\c)\\
R_2 &\text{ for }x\in[\c,\delta),
\end{cases}
\]
 with lengths $t_1+1$, $t_1$, $t_1+t_2$, $t_1+t_2+1$, $t_2$, respectively. We set $r_1=t_1$ and $r_2=t_2$.

\item Let $\d<\a<\c<\b$. We obtain
\[
R(x)=
\begin{cases}
R_1  &\text{ for }x\in[\gamma,\d)\\
R_1\omega_{B\to CA}(R_2)&\text{ for }x\in[\d,\a)\\
R_2R_1&\text{ for }x\in[\a,\c)\\
R_2 &\text{ for }x\in[\c,\b)\\
\omega_{B\to CA}(R_2) &\text{ for }x\in[\b,\delta),
\end{cases}
\]
with lengths $t_1$, $t_1+t_2+1$, $t_1+t_2$, $t_2$, $t_2+1$, respectively. We set $r_1=t_1$ and $r_2=t_2$.

\item Let $\b<\d<\a<\c$. We obtain
\[
R(x)=
\begin{cases}
\omega_{CA\to B}(R_1)  &\text{ for }x\in[\gamma,\b)\\
R_1&\text{ for }x\in[\b,\d)\\
R_1R_2&\text{ for }x\in[\d,\a)\\
R_2\omega_{CA\to B}(R_1) &\text{ for }x\in[\a,\c)\\
R_2 &\text{ for }x\in[\c,\delta),
\end{cases}
\]
 with lengths $t_1-1$, $t_1$, $t_1+t_2$, $t_1+t_2-1$, $t_2$, respectively. We set $r_1=t_1-1$ and $r_2=t_2$.

\item Let $\d<\b<\c<\a$. We obtain
\[
R(x)=
\begin{cases}
R_1  &\text{ for }x\in[\gamma,\d)\\
R_1\omega_{AC\to B}(R_2)&\text{ for }x\in[\d,\b)\\
R_2R_1&\text{ for }x\in[\b,\c)\\
R_2 &\text{ for }x\in[\c,\a)\\
\omega_{AC\to B}(R_2) &\text{ for }x\in[\a,\delta),
\end{cases}
\]
with lengths $t_1$, $t_1+t_2-1$, $t_1+t_2$, $t_2$, $t_2-1$, respectively. We set $r_1=t_1$ and $r_2=t_2-1$.

\item Let $\a<\d<\c<\b$. We obtain
\[
R(x)=
\begin{cases}
\omega_{B\to AC}(R_1)  &\text{ for }x\in[\gamma,\a)\\
R_1&\text{ for }x\in[\a,\d)\\
R_1\omega_{B\to CA}(R_2)&\text{ for }x\in[\d,\c)\\
R_2 &\text{ for }x\in[\c,\b)\\
\omega_{B\to CA}(R_2) &\text{ for }x\in[\b,\delta),
\end{cases}
\]
with lengths $t_1+1$, $t_1$, $t_1+t_2+1$, $t_2$, $t_2+1$, respectively. We set $r_1=t_1$ and $r_2=t_2$.

\item \label{t12:bdca} Let $\b<\d<\c<\a$. We obtain
\[
R(x)=
\begin{cases}
\omega_{CA\to B}(R_1)  &\text{ for }x\in[\gamma,\b)\\
R_1&\text{ for }x\in[\b,\d)\\
R_1\omega_{AC\to B}(R_2)&\text{ for }x\in[\d,\c)\\
R_2 &\text{ for }x\in[\c,\a)\\
\omega_{AC\to B}(R_2) &\text{ for }x\in[\a,\delta),
\end{cases}
\]
 with lengths $t_1-1$, $t_1$, $t_1+t_2-1$, $t_2$, $t_2-1$, respectively. We set $r_1=t_1-1$ and $r_2=t_2-1$.

\item Let $\d<\a<\b<\c$. We obtain
\[
R(x)=
\begin{cases}
R_1  &\text{ for }x\in[\gamma,\d)\\
R_1R_2&\text{ for }x\in[\d,\a)\\
\omega_{CA\to B}(R_2R_1)&\text{ for }x\in[\a,\b)\\
R_2R_1&\text{ for }x\in[\b,\c)\\
R_2 &\text{ for }x\in[\c,\delta),
\end{cases}
\]
 with lengths $t_1$, $t_1+t_2$, $t_1+t_2-1$, $t_1+t_2$, $t_2$, respectively. We set $r_1=t_1-1$ and $r_2=t_2$.

\item Let $\d<\b<\a<\c$. We obtain
\[
R(x)=
\begin{cases}
R_1  &\text{ for }x\in[\gamma,\d)\\
R_1R_2&\text{ for }x\in[\d,\b)\\
\omega_{B\to AC}(R_2R_1)&\text{ for }x\in[\b,\a)\\
R_2R_1&\text{ for }x\in[\a,\c)\\
R_2 &\text{ for }x\in[\c,\delta),
\end{cases}
\]
 with lengths $t_1$, $t_1+t_2$, $t_1+t_2+1$, $t_1+t_2$, $t_2$, respectively. We set $r_1=t_1$ and $r_2=t_2$.
\end{enumerate}
\pfk

\begin{remark}
When describing the $I$-itineraries using the words $R_1$, $R_2$, we could apply the rules of Lemma~\ref{l:jakitiner} in a different order.
By doing so, we would obtain the itineraries expressed differently, which yields interesting relations between words $R_1,R_2$.
For example, in the case (ix), we derive that the $I$-itinerary of $x\in[\d,\c)$ is $R(x)=R_1\omega_{B\to CA}(R_2)=R_2\omega_{B\to AC}(R_1)$.

Note also the symmetries between the cases (i) and (ii), (iii) and (iv), (v) and (vi), (vii) and (viii), in
consequence of Proposition \ref{p:symetrie}. Indeed, if
  we exchange  pair of points $ \d \leftrightarrow\c$,     $ \b \leftrightarrow\a$,  letters
   $A \leftrightarrow C $, and finally the inequalities $<$ and $>$, we obtain a symmetric
    situation in the list of  cases we discussed in the proof. In this sense, each of  cases  (ix) up to (xii) is symmetric to itself.
\end{remark}

\begin{example} \label{ex:vsech12}
Set $\alpha = \frac{1}{5} \sqrt{5} - \frac{1}{5}$ and $\beta = -\frac{1}{6} \sqrt{5} + \frac{2}{3}$.
Table~\ref{tab:vsech12} shows $12$ choices of $I = [\gamma, \delta)$ which produce $12$ distinct orders
of the points $\a$, $\b$, $\c$ and $\d$, shown in the third column. The last column contains the respective lengths of the $5$ distinct $I$-itineraries.

\vspace{\baselineskip}
\renewcommand*{\arraystretch}{1.5}
\begin{table}
\centering
%\resizebox{\textwidth}{!}{%
\begin{tabular}{l|c|c|c|c|c|c|c|c|c|c}
$\gamma$ & $\delta$ & type & lengths \\ \hline
$ \frac{6}{25} $ & $ \frac{99}{100} $ &
 $ \a < \b < \d < \c $ & $ \left[2, 1, 2, 3, 1\right] $ \\
 $ \frac{29}{100} $ & $ \frac{71}{100} $ &
 $ \d < \c < \a < \b $ & $ \left[1, 15, 14, 13, 14\right] $ \\
 $ \frac{77}{100} $ & $ \frac{4}{5} $ &
 $ \b < \a < \d < \c $ & $ \left[88, 89, 88, 109, 21\right] $ \\
$ \frac{7}{25} $ & $ \frac{3}{4} $ &
 $ \d < \c < \b < \a $ & $ \left[1, 13, 12, 13, 12\right] $ \\
 $ \frac{1}{100} $ & $ \frac{3}{4} $ &
 $ \a < \d < \b < \c $ & $ \left[2, 1, 2, 3, 1\right] $ \\
$ \frac{1}{100} $ & $ \frac{29}{100} $ &
 $ \d < \a < \c < \b $ & $ \left[2, 14, 13, 11, 12\right] $ \\
$ \frac{1}{4} $ & $ \frac{99}{100} $ &
 $ \b < \d < \a < \c $ & $ \left[1, 2, 3, 2, 1\right] $ \\
$ \frac{71}{100} $ & $ \frac{99}{100} $ &
 $ \d < \b < \c < \a $ & $ \left[2, 13, 14, 12, 11\right] $ \\
$ \frac{1}{25} $ & $ \frac{37}{50} $ &
 $ \a < \d < \c < \b $ & $ \left[2, 1, 4, 2, 3\right] $ \\
$ \frac{29}{100} $ & $ \frac{99}{100} $ &
 $ \b < \d < \c < \a $ & $ \left[1, 2, 4, 3, 2\right] $ \\
$ \frac{1}{100} $ & $ \frac{99}{100} $ &
 $ \d < \a < \b < \c $ & $ \left[1, 2, 1, 2, 1\right] $ \\
$ \frac{1}{4} $ & $ \frac{3}{4} $ &
 $ \d < \b < \a < \c $ & $ \left[1, 12, 13, 12, 11\right] $ \\
\end{tabular}
%}
\caption{The cases (i)--(xii) from the proof of Theorem~\ref{t:returntime} for $\alpha = \frac{1}{5} \sqrt{5} - \frac{1}{5}$, $\beta = -\frac{1}{6} \sqrt{5} + \frac{2}{3}$ as in Example~\ref{ex:vsech12}.
The endpoints of the interval $I=[\gamma,\delta)$ are in the first and second column.
The last column contains a list of lengths of $I$-itineraries of all $5$ subintervals of $I$ starting from the leftmost one.
}
\label{tab:vsech12}
\end{table}

Let us describe in detail one of the cases, namely the case $\b < \d < \c < \a$.
The induced interval is determined by setting $\gamma = \frac{29}{100}$ and $\delta = \frac{99}{100}$.
One can verify that
\begin{eqnarray*}
\b = T^{-0}(\beta) = \frac{1}{6} \sqrt{5} + \frac{1}{3} \approx 0.706011329583298; \\
\d = T^{-2}(\delta) \frac{11}{30} \sqrt{5} + \frac{37}{300} \approx 0.943224925083256; \\
\c = T^{-3}(\gamma) = \frac{8}{15} \sqrt{5} - \frac{73}{300} \approx 0.949236254666554; \\
\a = T^{-1}(\alpha) = \frac{11}{30} \sqrt{5} + \frac{2}{15} \approx 0.953224925083256.
\end{eqnarray*}

It corresponds to the case {\rm (\ref{t12:bdca})} in the proof of Theorem~\ref{t:returntime} with $R_1 = CA$ and $R_2 = CAC$.
The $I$-itinerary of a point $x \in I = [\gamma,\delta)$ is
\[
R(x)=
\begin{cases}
B &\text{ for }x\in[\gamma,\b)\\
CA &\text{ for }x\in[\b,\d)\\
CACB &\text{ for }x\in[\d,\c)\\
CAC &\text{ for }x\in[\c,\a)\\
CB &\text{ for }x\in[\a,\delta)
\end{cases}
\]
\end{example}

%%%%%%%%%%%%%%%%%%%%%%%%%%%%%%%%%%%%%%%%%%%%%%%%%%%%%%%%%%%%%%%%%%%%%%%%%%
\section{Gaps and distance theorems}\label{sec:Gap}

Let us reinterpret  the statement of
Theorem \ref{t:returntime} in point of view of three gap and  three distance theorems which are narrowly connected with exchange of two
intervals. Under the name three gap theorem one usually refers to the
description of gaps between neighbouring elements of the set
\[
{\mathcal G}(\alpha,\delta) := \big\{n\in\N \colon \{n\alpha\} < \delta \big\} \subset \N\,,
\]
where $\alpha\in \R\setminus\Q$, $\delta\in(0,1)$ and
$\{x\}=x-\lfloor x\rfloor$ stands for the fractional part of $x$,
see~\cite{3gap}. Sometimes one uses a more general formulation, namely
the set
\[
{\mathcal G}(\alpha,\rho,\gamma,\delta) := \big\{n\in\N : \gamma \leq \{n\alpha+\rho\} < \delta \big\} \subset \N\,,
\]
where moreover $\rho\in\N$, $0\leq \gamma<\delta< 1$. The three gap
theorem states that there exist integers
$r_1, r_2$ such that gaps between
neighbours in ${\mathcal G}(\alpha,\rho,\gamma,\delta)$ take at most
three values, namely in the set $\{r_1,r_2,r_1+r_2\}$.

Let us interpret the three gap theorem in the frame of exchange of two
intervals $J_0=[0,1-\alpha)$, $J_1=[1-\alpha,1)$. The transformation
$T:[0,1)\to[0,1)$ is of the form
\[
T(x) = \begin{cases}
x+\alpha & \text{ for }x\in[0,1-\alpha)\,,\\
x+\alpha-1 & \text{ for }x\in[1-\alpha,1)\,,
\end{cases}
\quad\text{i.e. }\quad
T(x) = \{x+\alpha\}\,.
\]
Therefore we can write
\begin{equation}\label{eq:3gapsturmian}
{\mathcal G}(\alpha,\rho,\gamma,\delta) := \big\{n\in\N \colon T^n(\rho)\in[\gamma,\delta)\big\}\,,
\end{equation}
and the gaps in this set correspond to return times to the interval $[\gamma,\delta)$ under the transformation $T$.

Our Theorem~\ref{t:returntime} is an analogue of the three gap
theorem in the form~\eqref{eq:3gapsturmian} generalized for the case
when the transformation $T$ is a non-degenerate 3iet. We see that
there are 5 gaps, but still expressed using two basic values
$r_1,r_2$.

The so-called three distance theorem focuses on distances between neighbours of the set
\[
{\mathcal D}(\alpha,\rho,N):=\big\{ \{\alpha n + \rho\} \colon n\in\N,\ n<N \big\} \subset [0,1)\,.
\]
The three distance theorem ensures existence of $\Delta_1,\Delta_2>0$
such that distances between neighbours in ${\mathcal
D}(\alpha,\rho,N)$ take at most three values, namely in
$\{\Delta_1,\Delta_2,\Delta_1+\Delta_2\}$.

In the framework of 2iet $T$, we can write for the distances
\begin{equation}\label{eq:3distancesturmian}
{\mathcal D}(\alpha,\rho,N):=\big\{  T^n(\rho) \colon n\in\N,\ n<N \big\} \subset [0,1)\,.
\end{equation}

We could try to study the analogue of the three distance theorem in
the form~\eqref{eq:3distancesturmian} for exchanges of three intervals. In fact, it can be derived
from the results of~\cite{GuMaPeBordeaux} that if $T$ is a 3iet with
discontinuity points $\alpha,\beta$, then
\begin{equation*}
{\mathcal D}(\alpha,\beta,\rho,N):=\big\{  T^n(\rho) \colon n\in\N,\ n<N \big\}
\end{equation*}
has again at most three distances $\Delta_1$, $\Delta_2$, and $\Delta_1+\Delta_2$ for some positive $\Delta_1,\Delta_2$.

The three distance theorem can also be used to derive that the frequencies of factors of length $n$ in a Sturmian word take at most three values.
Recall that the frequency of a factor $w$ in the infinite word ${\bf u}= u_0u_1u_2\ldots$ is given by
\[
{\rm freq}(w):=\lim_{N\to\infty}\frac1N\big(\#\{0\leq i< N \colon w \text{ is a prefix of } u_iu_{i+1}\ldots \}\big)\,,
\]
if the limit exists.

It is a well known fact that the frequencies of factors of length
$n$ in a coding of an exchange of intervals are given by the lengths
of cylinders corresponding to the factors.
% The cylinder of a factor $w$ is an intervals of points $\rho\in[0,1)$ whose coding has the same prefix $w$.
The boundary points of these cylinders are $T^{-j}(1-\alpha)$, for
$j=0,\dots,n-1$. Consequently, the
distances in the set ${\mathcal D}(\alpha,1-\alpha,N)$ are precisely
the frequencies of factors, and the three distance theorem implies the
well known fact that Sturmian words have for each $n$ only three
values of frequencies of factors of length $n$, namely
$\varrho_1,\varrho_2,\varrho_1+\varrho_2$.

The frequencies of factors of length $n$ in 3iet words are given by
distances between neighbours of the set
\begin{equation*}
\big\{  T^{-n}(\alpha) \colon n\in\N,\ n<N \big\}\cup \big\{  T^{-n}(\beta) \colon n\in\N,\ n<N \big\}\,.
\end{equation*}
In~\cite{BaPe} it is shown, based on the study of Rauzy graphs, that
the number of distinct values of frequencies in infinite words with
reversal closed language satisfies
\[
\#\{{\rm freq}(w) \colon w\in{\mathcal L}({\bf u}), |w|=n\} \leq 2 \Bigl({\mathcal C}_{\bf u}(n) - {\mathcal C}_{\bf u}(n-1)\Bigr) + 1\,,
\]
which in case of 3iet words reduces to $\leq 5$.
Paper~\cite{ferenczi} shows, that the set of integers $n$ for which
this bound is achieved, is of density 1 in $\N$.

%%%%%%%%%%%%%%%%%%%%%%%%%%%%%%%%%%%%%%%%%%%%%%%%%%%%%%%%%%%%%%%%%%%%%%%%%%
\section{Description of the case of three $I$-itineraries} \label{sec:3}

The cases (i) -- (xii) in the proof of Theorem~\ref{t:returntime} correspond to the generic
 instances of a subinterval $I$ in a non-degenerate 3iet which lead to
5 different $I$-itineraries. Let us focus on the cases where, on the contrary, the set of
 $I$-itineraries has only 3 elements.
First we recall two reasons why such cases are interesting.

  For a factor $w$ from the language of a non-degenerate 3iet transformation $T$, denote
\[
[w]=\{\rho \in [0,1 ) \colon  w \text{ is a prefix of }  {\bf u_\rho}\}\,.
\]
It is easy to see that $[w]$ -- usually called the cylinder of $w$ -- is a semi-closed interval and its
 boundaries belong to the set $ \{ T^{-i}(z) \colon 0 \leq i < n, z \in
\{\alpha, \beta\}\}$. Clearly, a factor $v$ is a return word to the factor $w$ if and only if
 $v$ is a $[w]$-itinerary. It is well known \cite{vuillon} that any factor of an  infinite word coding a
  non-degenerate 3iet has exactly three return words and thus the set $It_{[w]}$ has three elements.

The second  reason why to study intervals  $I$  yielding
 three $I$-itineraries is that any morphism fixing a non-degenerate 3iet word corresponds
to such an interval $I$. Details of this correspondence  will be explained  in Section~\ref{sec:substitutions}.

\begin{proposition}\label{p:3itinerare}
Let $T$ be a non-degenerate 3iet and let $I=[\gamma,\delta)\subset [0,1)$ be such that $\#{\rm It}_I = 3$. One of the following cases occurs:
\begin{enumerate}[(i)]
\item \label{l:jakitiner_3rw_1} $\b = \d < \a=\c$ and
\[
R(x) =
\begin{cases}
R_1  &\text{ for }x\in[\gamma,\b)\\
\omega_{B\to CA}(R_1R_2) = \omega_{B\to AC}(R_2R_1) & \text{ for } x\in[\b,\a)\\
R_2 &\text{ for }x\in[\a,\delta)
\end{cases}
\]
\item \label{l:jakitiner_3rw_2} $\a = \d < \b=\c$ and
\[
R(x) =
\begin{cases}
R_1  &\text{ for }x\in[\gamma,\a)\\
\omega_{AC\to B}(R_1R_2) = \omega_{CA\to B}(R_2R_1) & \text{ for } x\in[\a,\b)\\
R_2 &\text{ for }x\in[\b,\delta)
\end{cases}
\]
\item \label{l:jakitiner_3rw_3} $\b < \a=\c=\d$ and
\[
R(x) =
\begin{cases}
\omega_{CA\to B}(R_1)  &\text{ for }x\in[\gamma,\b)\\
R_1 & \text{ for } x\in[\b,\a)\\
R_2 &\text{ for }x\in[\a,\delta)
\end{cases}
\]
%\textcolor[rgb]{1.00,0.00,0.00}{ PREMISTENO KVULI SYMTERII  Case iv) a vi)}
\item \label{l:jakitiner_3rw_6} $\b=\c=\d < \a$ and
\[
R(x) =
\begin{cases}
R_1 &\text{ for }x\in[\gamma,\a)\\
R_2 & \text{ for } x\in[\a,\b)\\
\omega_{AC\to B}(R_2)  &\text{ for }x\in[\b,\delta)
\end{cases}
\]
\item \label{l:jakitiner_3rw_4} $\a=\c=\d < \b$ and
\[
R(x) =
\begin{cases}
R_1 &\text{ for }x\in[\gamma,\a)\\
R_2 & \text{ for } x\in[\a,\b)\\
\omega_{B\to CA}(R_2)  &\text{ for }x\in[\b,\delta)
\end{cases}
\]
\item \label{l:jakitiner_3rw_5} $\a < \b=\c=\d$ and
\[
R(x) =
\begin{cases}
\omega_{B\to AC}(R_1)  &\text{ for }x\in[\gamma,\a)\\
R_1 & \text{ for } x\in[\a,\b)\\
R_2 &\text{ for }x\in[\b,\delta)
\end{cases}
\]

\end{enumerate}
\end{proposition}

\begin{proof}[Sketch of a proof]
Since by Lemma~\ref{l:keane} the subintervals
of $I$ corresponding to the same itinerary are delimited by the points $\a,\b,\c$ and $\d$,
we may have $\#{\rm It}_I=3$ only if some of these points coincide, more precisely if  $\# \{\a,\b,\c,\d\} = 2$.
The non-degeneracy of the considered 3iet implies that always $\a\neq\b$, which further limits the discussion.

The six cases listed in the statement are the possibilities of how this may happen,
 respecting the condition $\d < \c$ or  $\d = \c$.
In order to describe the itineraries, denote again
\[
R_1=R(\d-\varepsilon) \ \text{ and } \ R_2=R(\c+\varepsilon)
\]
for $\varepsilon>0$ sufficiently small.
One can then follow the ideas of the proof of Lemma~\ref{l:jakitiner}.
\end{proof}

%3 jsou kdyz return words ke slovum -- popis u Ferenczi, Holton, Zamboni
%
%anebo kdyz je indukovane homoteticke puvodnimu... Substitucni invariance
%

%Remark~\ref{re:sedi_na_kraji}

Let us apply Proposition~\ref{p:3itinerare} in order to provide the description of
 return words to factors of a non-degenerate 3iet word. If a factor $w$ has a unique right prolongation in the language $\mathcal{L}(T)$,
  i.e. there exists only one letter $a \in \mathcal{A}$ such that $wa \in\mathcal{L}(T)$, then the set of
   return words to $w$  and the set of return words to $wa$ coincide. And (almost) analogously, if a
   factor $w$ has a unique left prolongation in the language $ \mathcal{L}(T)$, say $aw$ for some $a \in \mathcal{A}$,
     then a word $v$ is a return word to $w$ if and only if $ava^{-1}$ is a return word to  $aw$. Consequently,
      to describe the structure of return words to a given factor $w$, we can restrict to factors which have at
       least two right and at least two left prolongations. Such factors are called {\it bispecial}.  It is readily
       seen that the language of an aperiodic recurrent infinite word ${\bf u}$ contains infinitely many bispecial factors.
Before giving the description of return words to bispecial factors, we state the following lemma.

\begin{lemma} \label{l:zrcadlovy_cylindr}
Let $w$ belong to language  of  a non-degenerate 3iet $T$. Denote $n=|w|$. For the cylinder of its reversal $\overline{w}$, one has
\[
[\overline{w}] = \overline{T^{n}\left( [w]\right)}\,.
\]
\end{lemma}

\begin{proof}  According to definition of $[w]$, for each  $[w]$-itinerary $r$,  the word $rw$ belongs to the language
 and $w$ occurs in  $rw$  exactly twice, as a prefix and as a suffix. In other words   $r$ is a return word to $w$.
  Moreover,   $[w]$ is the maximal  (with respect to  inclusion) interval with this property.
According to Remark~\ref{pozn:skakani}, if $r$ is an $[w]$-itinerary, then  the word $w^{-1}rw$ is an $T^{n}([w])$-itinerary.
Applying  Proposition~\ref{p:symetrie} to the interval $T^{n}([w])$ we obtain that $s:=\overline{ w^{-1}rw}$  is an
 $\overline{T^{n}([w])}$-itinerary. Since the word $s \overline{w} = \overline{rw}$   has a prefix $\overline{w}$ and
  a suffix $\overline{w}$,  with no other occurrences of $\overline{w}$, the word $s$ is a return word to
   $\overline{w}$ and thus  by definition of the cylinder, $s=\overline{ w^{-1}rw}$  belongs to
    $ [\overline{w}]$-itinerary for any $\overline{T^{n}([w])}$-itinerary $s$.  From the maximality of
    the cylinder we have  $\overline{T^{n}([w])}\subset [\overline{w}]$.  Since lengths of the intervals
     $[w]$ and  $T^{n}([w])$ coincide we  have,  in particular, that  the length of interval $[w]$ is
      less or equal to the length of the interval $[\overline{w}]$. But from the symmetry of the role
      $w$ and $\overline{w}$, their length must be equal and thus  $\overline{T^{n}([w])}=[\overline{w}]$.
\end{proof}

The language of $T$  contains  two types of bispecial factors:
palindromic and non-palindromic.   In \cite{FeHoZa},    Ferenczi,
Holton and Zamboni studied the structure of return words to
non-palindromic bispecial factors. The following proposition
completes this description.

\begin{proposition}
Let $w$ be a bispecial factor.
If $w$ is a palindrome, then its return words are described by the cases (\ref{l:jakitiner_3rw_1}) and (\ref{l:jakitiner_3rw_2}) of Proposition~\ref{p:3itinerare}.
If $w$ is not a palindrome, then its return words are described by the cases (\ref{l:jakitiner_3rw_3}) -- (\ref{l:jakitiner_3rw_6}) of Proposition~\ref{p:3itinerare}.
\end{proposition}

\begin{proof}
Let $w$ be a bispecial factor.
If $w$ is not a palindrome, the claim follows from Theorem 4.6 of \cite{FeHoZa}.

Assume $w$ is a palindrome and let $[w] = [T^{-\ell}(L) , T^{-r}(R) )$ with $L,R \in \{\alpha, \beta\}$ and $0 \leq \ell,r < |w|$.
By Lemma~\ref{l:zrcadlovy_cylindr} we have $[\overline{w}] = \overline{T^{|w|}\left([ w]\right)}$.
Since $w=\overline{w}$, we have $I_{w} = I_{\overline{w}}$, and thus
\[
I_w = [T^{-\ell}(L), T^{-r}(R)) =  [1-T^{n-r}(R), 1- T^{n-\ell}(L)) = I_{\overline{w}}.
\]
Since the considered 3iet is non-degenerate,
the parameters $\alpha, \beta$ satisfy \eqref{eq:nondeg}. Consequently, the equation $T^{-\ell}(L) = 1-T^{n-r}(R)$ has a solution if and only if $R \neq L$.
Thus, we have neither $\a = \c = \d$ nor $\b = \c = \d$ and we are in the case (\ref{l:jakitiner_3rw_1})
or (\ref{l:jakitiner_3rw_2}) of Proposition~\ref{p:3itinerare}.
\end{proof}

%%%%%%%%%%%%%%%%%%%%%%%%%%%%%%%%%%%%%%%%%%%%%%%%%%%%%%%%%%%%%%%%%%%%%%%%%%
\section{Substitution invariance and conjugation of substitutions}\label{sec:substitutions}

Let us recall the relation of induction to a subinterval $I$ to substitution invariance of 3iet words.
Let $I$ be an interval $I\subset[0,1)$ such that the set ${\rm It}_I$ of $I$-itineraries has three elements, say $R_A$, $R_B$ and $R_C$.
For every $\rho\in I$, the infinite word ${\bf u}_\rho$ can be written as a concatenation of words $R_A$, $R_B$ and $R_C$.
For a letter $Y\in\{A,B,C\}$ denote $I_Y=\{x\in I \colon R(x)=R_Y\}$. Obviously, $I=I_A\cup I_B\cup I_C$, and the induced mapping $T_I$ is an exchange of these three intervals.
The order of the words $R_A, R_B$ and $R_C$ in the concatenation is determined by the iterations of $T_I(\rho)$.

Suppose that $T_I$ is homothetic to $T$. Recall that mappings $f:I_f\to I_f$ and $g:I_g\to I_g$ are homothetic
if there exists an affine bijection $\Phi:I_f\to I_g$ with $\Phi(x) = \lambda x+\mu$ such that
\begin{equation}\label{eq:afbij}
\Phi f(x) = g \Phi(x)\quad \text{ for all } x\in I_f\,.
\end{equation}
This means that $f$ and $g$ behave in the same way, up to a scaling factor $\lambda$ and a shift $\mu$ of
the domains $I_f$ and $I_g$. In other words, the graphs of the mappings
$f$ and $g$ are the same, up to their scale and placing.
The homothety of $T$ and $T_I$ implies that $\Phi(J_Y)=I_Y$ for all
$Y \in \{A,B,C\}$. From~\eqref{eq:afbij}, we derive for every
$k\in\N$ that $\Phi T^k(x)=T^k_I\Phi(x)$ for $x\in[0,1)$, and thus
$\Phi T^k(\rho)=T^k_I(\rho)$ whenever
\begin{equation}\label{eq:homofixesintercept}
\Phi(\rho)=\rho\,,
\end{equation}
i.e., $\rho$ is the homothety center. From the relation
$\Phi(J_Y)=I_Y$ it follows that the $k$-th element in the
concatenation of itineraries $R_A$, $R_B$ and $R_C$ is equal to
$R_Y$ if and only if the $k$-th letter in the infinite word ${\bf
u}_\rho$ is equal to $Y$. This is equivalent to saying that the
infinite word ${\bf u}_\rho$ is invariant under the substitution
$\eta$ given by
\begin{equation}\label{eq:rabc}
\eta(A)=R_A,\ \eta(B)=R_B,\ \eta (C)=R_C\,.
\end{equation}

We conclude that the existence of an interval $I$ with three
itineraries and $T_I$ homothetic to $T$ leads to a substitution
fixing a 3iet word whose intercept is the homothety center $\rho$.
In fact, the converse holds, too, as shown in~\cite{ArBeMaPe}. We
summarize both statements as follows.

\begin{theorem}[\cite{ArBeMaPe}]\label{thm:invhom}
Let $\xi$ be a primitive substitution over $\{A,B,C\}$ with
incidence matrix $M$ and let $T$ be a non-degenerate 3iet. The
substitution $\xi$ fixes the word ${\bf u}_\rho$ coding the orbit of
a point $\rho\in[0,1)$ under $T$ if and only if there exists an
interval $I\subset[0,1)$ with $I$-itineraries
$\mathit{It}_I=\{R_A,R_B,R_C\}$ such that $T_I$ is homothetic to
$T$, $\rho$ is the homothety center, and the substitution $\eta$
given by
\[
\eta=
\begin{cases}
\xi &\text{if no eigenvalue of $M$ belongs to }(-1,0),\\
\xi^2 &\text{otherwise,}
\end{cases}
\]
satisfies $\eta(A)=R_A$, $\eta(B)=R_B$ and $\eta (C)=R_C$.
\end{theorem}

Let us mention that the scaling factor $\lambda\in(0,1)$ in the
homothety mapping
$\Phi (x) =\lambda x+\mu$ is equal to the length of
the interval $I=[\gamma,\delta)$, i.e., $\lambda=\delta-\gamma$, and
the shift $\mu$ is equal to the left end-point of the interval $I$,
namely $\gamma$. Moreover, it is related to the intercept $\rho$ of
an infinite word  ${\bf u}_\rho$ in the following way: one has
$\mu=\gamma=\rho(1-\lambda)$, as follows
from~\eqref{eq:homofixesintercept}. In fact, $\lambda$ is an
eigenvalue of the incidence matrix of $\eta$. It follows
from~\cite{ArBeMaPe} that if $\xi$ has such an eigenvalue, then the
choice $\eta=\xi$ is sufficient. Otherwise, the incidence matrix of
$\xi^2$ has such an eigenvalue.

By Theorem~\ref{thm:invhom}, if ${\bf u}_\rho$ is invariant under a
substitution, we find an interval $I$ such that $T_I$ is homothetic
to $T$. If $I'=T(I)$ is again an interval, then $T_{I'}$ is also
homothetic to $T$, and the $I'$-itineraries change with respect to
the $I$-itineraries, as described in Remark~\ref{pozn:skakani}. To
show the relation of the corresponding substitutions, we need the
following definition.

\begin{definition}\label{d:conjugmorf}
Let $\varphi$ and $\psi$ be morphisms over $\A^*$ and let $w\in\A^*$
be a word such that $w\varphi(a) = \psi(a)w$ for every letter
$a\in\A$. The morphism $\varphi$ is said to be a \textit{left
conjugate} of $\psi$ and $\psi$ is said to be a \textit{right
conjugate} of $\varphi$. We write $\varphi\vartriangleleft\psi$. If
$\varphi$ is a left or right conjugate of $\psi$, then we say
$\varphi$ is \textit{conjugate} to $\psi$. If the only left
conjugate of $\varphi$ is $\varphi$ itself, then $\varphi$ is called
the \textit{leftmost conjugate} of $\psi$ and we write
$\varphi=\psi_L$. If the only right conjugate of $\psi$ is $\psi$
itself, then $\psi$ is called the \textit{rightmost conjugate} of
$\varphi$ and we write $\psi=\varphi_R$.
\end{definition}

Note that given a substitution $\xi$, its leftmost and rightmost conjugates $\xi_L$ and $\xi_R$ may not exist.
If this happens, it can be shown that its fixed point is a periodic word.
All the substitutions considered here thus possess their leftmost and rightmost conjugates.

\begin{proposition}\label{p:conjug}
Let ${\bf u}_\rho$ be a 3iet word coding the orbit of the point
$\rho \in [0,1)$ under a non-degenerate 3iet $T$. Moreover, assume
that ${\bf u}_\rho$ is a fixed point of a primitive substitution
$\eta$ such that the corresponding interval $I$ of
Theorem~\ref{thm:invhom} is of length $\lambda$. Let $\eta'$ be a
left conjugate of $\eta$, i.e., $\eta(a)w = w\eta'(a)$ for some word
$w\in\A^*$. The morphism $\eta'$ fixes the infinite word ${\bf
u}_{\rho'}$ with $\rho'$ satisfying
\begin{equation}\label{eq:conjugintercept}
(1-\lambda)\rho' = T^{n}\big((1-\lambda)\rho\big)\,,\quad\text{ where } n=|w|.
\end{equation}
Moreover, the interval $I'$ corresponding to $\eta'$ by Theorem~\ref{thm:invhom} satisfies $I'=T^n(I)$.
\end{proposition}

\pfz Suppose that $w$ is a letter, i.e., $w=X\in\A$. Necessarily,
the words $\eta(a)$ start with the letter $X$ for all $a\in\A$. This
means for the interval $I$ that $I\subset J_X$. According to
Remark~\ref{pozn:skakani}, the interval $I'=T(I)$ has three
$I'$-itineraries. Moreover, the induced mapping $T_{I'}$ is also
homothetic to $T$. Denote $I=[\gamma,\delta)$. The homothety between
the transformations $T$ and $T_I$ is achieved by the map
$\Phi(x)=\lambda x + \gamma$. The homothety between $T$ and $T_{I'}$
is the map $\Phi'(x)=\lambda x + T(\gamma)$. Since the intercepts
$\rho$ and $\rho'$ are by~\eqref{eq:homofixesintercept} fixed by the
homotheties $\Phi$ and $\Phi'$, respectively, we have
\begin{equation}\label{eq:cislo}
\Phi(\rho)=\lambda\rho + \gamma = \rho\, \quad\text{ and }\quad \Phi'(\rho')=\lambda\rho' + T(\gamma) = \rho'\,.
\end{equation}
Eliminating $\gamma$, we obtain
\[
(1-\lambda)\rho' = T(\gamma) = T\big((1-\lambda)\rho\big)\,.
\]
Since conjugation by any word $w$ can be performed letter by letter, the proof is finished.
\pfk

\begin{remark}\label{pozn:intercept}
Note that the first of equalities in~\eqref{eq:cislo} implies for the left boundary point $\gamma$ of the interval $I$ that
$\gamma=(1-\lambda)\rho$. If $w$ is as in Proposition~\ref{p:conjug}, then for for $0\leq k< n=|w|$, the iterations $T^k(I)$ are all intervals, and hence the
coding ${\bf u}_x$ of every point $x\in I$ starts with the same prefix $w$.
\end{remark}

In the following, we will also need to see the relation of the substitution $\eta$
corresponding to the interval $I=[\gamma,\delta)$ with the
substitution corresponding to the interval
$\overline{I}=[1-\delta,1-\gamma)$. It turns out that it is the
mirror substitution of $\eta$, defined in general as follows. For a
morphism $\xi: \A \to \A$, we define the morphism $\overline{\xi}:
\A \to \A$ by $\overline{\xi}(a)=\overline{\xi(a)}$ for $a\in\A$.

\begin{proposition}\label{p:mirror}
Let $I\subset[0,1)$ be a left-closed right-open interval such that
$\#\mathit{It}_I=3$ and $T_I$ is an exchange of three intervals with
the permutation $(321)$. The interval $\overline{I}$ satisfies
$\#\mathit{It}_{\overline{I}}=3$ and the induced map
$T_{\overline{I}}$ is homothetic to $T_I$. If, moreover, $T_I$ is
homothetic to $T$ and the substitution  $\eta$ corresponding to $I$
fixes the infinite word ${\bf u}_\rho$, then the substitution
corresponding to $\overline{I}$ is $\overline{\eta}$ and fixes the
infinite word ${\bf u}_{\overline{\rho}}$, where
$\overline{\rho}=1-\rho$.
\end{proposition}

\pfz
Denote $\mathit{It}_I=\{R_1,R_2,R_3\}$ and  $I_j=\{x\in I \colon R_I(x)=R_j\} = [\gamma_j,\delta_j)$ for $j=1,2,3$ so that $I_1<I_2<I_3$.
By Proposition~\ref{p:symetrie}, the $\overline{I}$-itineraries are $\overline{R_1},\overline{R_2}$ and $\overline{R_3}$, where
\[
I'_j=\{x\in \overline{I} \colon R_{\overline{I}}(x)=\overline{R_j}\} = [1-\delta'_j,1-\gamma'_j)\,,
\]
where $[\gamma'_j,\delta'_j)=T_I[\gamma_j,\delta_j)$. Since $T_I$ is an exchange of three intervals with the permutation $(321)$ we have
\[
T_I[\gamma_1,\delta_1)>T_I[\gamma_2,\delta_2)>T_I[\gamma_3,\delta_3)\,,
\]
and therefore $I'_1<I'_2<I'_3$. The induced map $T_{\overline{I}}$
is therefore an exchange of three intervals $I'_1,I'_2$ and $I'_3$
with the permutation $(321)$ and since $|I'_j|=|I_j|$ for $j=1,2,3$,
the transformation $T_{\overline{I}}$ is homothetic to $T_I$.

Suppose that $T_I$ is homothetic to the original 3iet $T$. By
Theorem~\ref{thm:invhom}, there is a substitution $\eta$
corresponding to the interval $I$ and satisfying $\eta(A)=R_1$,
$\eta(B)=R_2$ and $\eta(C)=R_3$. The mapping $T_{\overline{I}}$ is
homothetic to $T_I$ and thus also to $T$, the corresponding
substitution $\eta'$ satisfies $\eta'(A)=\overline{R_1}$,
$\eta'(B)=\overline{R_2}$ and $\eta'(C)=\overline{R_3}$. We can see
that $\eta'=\overline{\eta}$.

Let $\rho$ be the intercept of the infinite word which is fixed by
the substitution $\eta$. It is the center of homothety between $T_I$
and $T$, i.e., it is the fixed point of the mapping
$\Phi(x)=(\delta-\gamma)x+\gamma$. We have
$\rho=(\delta-\gamma)\rho+\gamma$, which implies
\[
\rho=\frac{\gamma}{1-\delta+\gamma}\,.
\]
Similarly, the intercept $\overline{\rho}$ of the substitution $\overline{\eta}$ satisfies $\overline{\rho}=(\delta-\gamma)\overline{\rho}+1-\delta$, whence
\begin{equation*}
\overline{\rho}=\frac{1-\delta}{1-\delta+\gamma}=1-\rho\,. \eqno{\qedhere}
\end{equation*}
\pfkNoQed

For a finite word $w$, ${\rm Fst}(w)$ and ${\rm Lst}(w)$ denote the first and last letters of $w$, respectively.

\begin{remark}\label{pozn:uspor}
Let $\eta$ be a primitive substitution given by
Theorem~\ref{thm:invhom} fixing a 3iet word. Necessarily, the first
and the last letters of $\eta(A)$, $\eta(B)$ and $\eta(C)$ satisfy
\[
\begin{aligned}
{\rm Fst}\big(\eta(A)\big) \leq {\rm Fst}\big(\eta(B)\big) \leq {\rm Fst}\big(\eta(C)\big)  \text{ and }  %\,, %\\
{\rm Lst}\big(\eta(A)\big) \leq {\rm Lst}\big(\eta(B)\big) \leq {\rm Lst}\big(\eta(C)\big),
\end{aligned}
\]
where we consider the order $A < B < C$. The inequalities for the
first letters follow from the definition of an exchange of
intervals, namely from the fact that the words $\eta(A)$, $\eta(B)$
and $\eta(C)$ are given as $I$-itineraries. By
Proposition~\ref{p:mirror}, the last letters of the words $\eta(A)$,
$\eta(B)$ and $\eta(C)$ are the first letters of the words
$\overline{\eta(A)}$, $\overline{\eta(B)}$ and $\overline{\eta(C)}$
which proves the second set of inequalities.
\end{remark}

%%%%%%%%%%%%%%%%%%%%%%%%%%%%%%%%%%%%%%%%%%%%%%%%%%%%%%%%%%%%%%%%%%%%%%%%%%
\section{Ternarization} \label{sec:ternarizace}

A characterization of 3iet words over the alphabet $\{A,B,C\}$ by morphic images of Sturmian words over $\{0,1\}$ is derived in~\cite{ArBeMaPe}.
Let $\sigma_{01}$ and $\sigma_{10}$  be morphisms $\{A,B,C\}\to\{0,1\}$ defined by
\[
\begin{aligned}
\sigma_{01}(A)&=\sigma_{10}(A)=0\,,\\
\sigma_{01}(B)&=01,\ \sigma_{10}(B)=10\,,\\
\sigma_{01}(C)&=\sigma_{10}(C)=1\,.
\end{aligned}
\]

\begin{theorem}[\cite{ArBeMaPe}]
An infinite word ${\bf u}\in\{A,B,C\}^{\N}$ is a 3iet word if and only if $\sigma_{01}({\bf u})$ and $\sigma_{10}({\bf u})$ are Sturmian words.
\end{theorem}

This theorem was an important tool in~\cite{AmMaPe} for the description of substitutions $\eta$ from Theorem~\ref{thm:invhom} fixing a 3iet word.
Since this result is important for our further considerations, we cite it as Theorem~\ref{thm:ternarizace}, but first, we need some definitions.

\begin{definition}\label{d:amicablewords}
  Let $u$ and $v$ be finite or infinite words over the alphabet $\{0,1\}$. We say that $u$ is \textit{amicable} to
  $v$, and denote it by $u\propto v$, if there exists a ternary word $w$ over $\{A,B,C\}$ such that
  $u=\sigma_{01}(w)$ and $v=\sigma_{10}(w)$.
  In such case, we set $w:=\mathrm{ter}(u,v)$ and say that $w$ is the
  \textit{ternarization} of $u$ and $v$.
\end{definition}

\begin{definition}\label{d:amicablemorphisms}
  Let $\varphi,\psi:\{0,1\}^*\rightarrow\{0,1\}^*$ be two morphisms.
  We say that $\varphi$ is \textit{amicable} to $\psi$, and denote it by
  $\varphi\propto\psi$, if the three following
  relations hold
  \begin{equation}\label{eq:amicmorf}
  \begin{split}
    \varphi(0)&\propto\psi(0)\,,\\
    \varphi(1)&\propto\psi(1)\,,\\
    \varphi(01)&\propto\psi(10)\,.
  \end{split}
  \end{equation}
  The morphism $\eta:\{A,B,C\}^*\to\{A,B,C\}^*$ given by
\begin{align*}
\eta(A) & :=\mathrm{ter}(\varphi(0),\psi(0))\,, \\
\eta(B) & :=\mathrm{ter}(\varphi(01),\psi(10))\,, \\
\eta(C) & :=\mathrm{ter}(\varphi(1),\psi(1))\,,
\end{align*}
is called the ternarization of $\varphi$ and $\psi$ and denoted by
$\eta:=\mathrm{ter}(\varphi,\psi)$.
\end{definition}

\begin{remark}\label{pozn:amicstejnamatice}
If $u$ and $v$ is a pair of amicable words over $\{0,1\}$, then
$|u|_0=|v|_0$ and $|u|_1=|v|_1$. Consequently, if $\varphi$ and
$\psi$ are two amicable morphisms, then they have the same incidence
matrix.
\end{remark}

\begin{theorem}[\cite{AmMaPe}]\label{thm:ternarizace}
Let $\eta$ be a primitive substitution from Theorem~\ref{thm:invhom}
fixing a non-degenerate 3iet word ${\bf u}$. There exist Sturmian
morphisms $\varphi$ and $\psi$ having fixed points such that
$\varphi\propto\psi$ and $\eta={\rm ter}(\varphi,\psi)$. On the
other hand, if $\varphi$ and $\psi$ are Sturmian morphisms with
fixed points such that $\varphi\propto\psi$, then the morphism
$\eta={\rm ter}(\varphi,\psi)$ has a 3iet fixed point.
\end{theorem}

\begin{example}\label{ex:ternarizacesubst}
Consider the following Sturmian morphisms $\varphi,\psi:\{0,1\}^*\to\{0,1\}^*$,
\[
\begin{aligned}
\varphi(0)&=0110101\\
\varphi(1)&=01101
\end{aligned}\,,
\qquad
\begin{aligned}
\psi(0)&=1010101\\
\psi(1)&=10101
\end{aligned}\,.
\]
We verify the condition given in~\eqref{eq:amicmorf} and in the same time construct the ternarization $\eta=\mathrm{ter}(\varphi,\psi)$.
We check that $\varphi(0)\propto\psi(0)$,
\[
\begin{gathered}
  \varphi(0) = \quad\\ \psi(0) =\quad \\[1mm]  \eta(A) =\quad
\end{gathered}
\bAC\ \bC\ \bA\ \bC\ \bA\ \bC\,,
\]
and $\varphi(1)\propto\psi(1)$
\[
\begin{gathered}
  \varphi(1) = \quad\\ \psi(1) =\quad \\[1mm]  \eta(C) =\quad
\end{gathered}
\bAC\ \bC\ \bA\ \bC\,,
\]
and lastly that $\varphi(01)\propto\psi(10)$
\[
\begin{gathered}
  \varphi(01) = \quad\\ \psi(10) =\quad \\[1mm]  \eta(B) =\quad
\end{gathered}
\bAC\ \bC\ \bA\ \bC\ \bAC\ \bAC\ \bC\ \bA\ \bC\,.
\]
We obtained a ternarization of a pair of amicable Sturmian morphisms. By Theorem~\ref{thm:ternarizace}, $\eta$ fixes a 3iet word.
\end{example}

Theorem~\ref{thm:ternarizace} expresses the relation between
substitutions fixing 3iet words and Sturmian morphisms. Recall that
by a result of~\cite{seebold}, all Sturmian morphisms with the same
incidence matrix
$M=\big(\begin{smallmatrix}a&b\\c&d\end{smallmatrix}\big)$ can be
ordered by the relation $\vartriangleleft$ of conjugation into a
chain
\begin{equation}\label{eq:sturmorf}
\xi_1 \vartriangleleft \xi_2 \vartriangleleft \cdots \vartriangleleft \xi_N\,,\quad\text{ where } N=a+b+c+d-1\,.
\end{equation}
This implies that for every $i$ and $j$ such that $1\leq i<j\leq N$,
there exists a word $u\in\{0,1\}^*$ of length $j-i$ such that
$u\xi_i(a)=\xi_j(a)u$ for $a\in\{0,1\}$.

\begin{lemma}\label{l:sturmconjug}
Let $\eta={\rm ter}(\varphi,\psi)$, $\eta'={\rm ter}(\varphi',\psi')$, where $\varphi,\psi$ and $\varphi',\psi'$ are pairs
of amicable Sturmian morphisms over the alphabet $\{0,1\}$.
If $\eta\vartriangleleft\eta'$, then $\varphi\vartriangleleft\varphi'$ and $\psi\vartriangleleft\psi'$, and, moreover,
$\varphi=\xi_i$, $\psi=\xi_j$, $\varphi'=\xi_{i'}$, $\psi'=\xi_{j'}$ where $j-i=j'-i'$.
\end{lemma}

\pfz Since $\eta\vartriangleleft\eta'$, there exists a word
$w\in\{A,B,C\}^*$ such that $w\eta_1(X) = \eta_2(X)w$ for every
$X\in\{A,B,C\}$. We will show that then there exists an amicable
pair of words $u,v\in\{0,1\}^*$ with $|u|=|v|$ such that $w={\rm
ter}(u,v)$ and
\begin{equation}\label{eq:sturmconjug}
\begin{aligned}
u\varphi(b)&=\varphi'(b)u\\
v\psi(b)&=\psi'(b)v
\end{aligned}
\quad \text{for $b\in\{0,1\}$.}
\end{equation}
It suffices to prove this statement for $w$ of length 1, i.e., ${\rm
Fst} \left( \eta'(X) \right )=w$ for $X\in\{A,B,C\}$. If $w=A$, then
necessarily ${\rm Fst}\big(\varphi'(b)\big)={\rm
Fst}\big(\psi'(b)\big)={\rm Lst}\big(\varphi(b)\big)={\rm
Lst}\big(\psi(b)\big)=0$ for $b\in\{0,1\}$. Therefore $u=v=0$. If
$w=C$, then similarly, ${\rm Fst}\big(\varphi'(b)\big)={\rm
Fst}\big(\psi'(b)\big)={\rm Lst}\big(\varphi(b)\big)={\rm
Lst}\big(\psi(b)\big)=1$ for $b\in\{0,1\}$, and thus $u=v=1$. If
$w=B$, then $\varphi'(0)$ and $\varphi'(1)$ have prefix $01$, and
$\psi'(0)$ and $\psi'(1)$ have prefix $10$. Thus $u=01$, $v=10$ and
clearly, $w={\rm ter}(u,v)$.

Now $\varphi$ and $\psi$ are amicable Sturmian morphisms with the
same incidence matrix $M$, and since $\varphi',\psi'$ are their
conjugates, they also have the same incidence matrix, thus
$\varphi=\xi_i$, $\psi=\xi_j$, $\varphi'=\xi_{i'}$, $\psi'=\xi_{j'}$
for some $1\leq i,j,i',j'\leq N$. The relation $|u|=j-i=j'-i'$
follows from~\eqref{eq:sturmconjug}. \pfk

\begin{lemma}\label{l:sedispravne}
Let $\eta$ be a primitive substitution given by Theorem~\ref{thm:invhom} fixing a 3iet word. We have
\begin{equation*}\label{eq:abb}
\begin{split}
\Big({\rm Fst}\big(\eta_L(A)\big),{\rm Fst}\big(\eta_L(B)\big),{\rm Fst}\big(\eta_L(C)\big)\Big) & = \\
 \Big({\rm Lst}\big(\eta_R(A)\big),{\rm Lst}\big(\eta_R(B)\big),{\rm Lst}\big(\eta_R(C)\big)\Big) &= (A,B,B)\,,
 \end{split}
\end{equation*}
or
\begin{equation*}\label{eq:bbc}
\begin{split}
\Big({\rm Fst}\big(\eta_L(A)\big),{\rm Fst}\big(\eta_L(B)\big),{\rm Fst}\big(\eta_L(C)\big)\Big) &= \\
 \Big({\rm Lst}\big(\eta_R(A)\big),{\rm Lst}\big(\eta_R(B)\big),{\rm Lst}\big(\eta_R(C)\big)\Big) &= (B,B,C)\,.
\end{split}
\end{equation*}
\end{lemma}

\pfz Since the words $\eta(A)$, $\eta(B)$ and $\eta(C)$ are
$I$-itineraries for some interval $I$, the first letters of
$\eta(A)$, $\eta(B)$ and $\eta(C)$ cannot all be distinct. On the
contrary, suppose that the discontinuity points $\alpha$ and $\beta$
of the transformation $T$ belong to the interval $I$.
It implies that these points coincide with the discontinuity points $\d$ and $\c$ of the induced map $T_I$.
But this means that $T_I$ is not homothetic to $T$, which is a contradiction.

By Remark~\ref{pozn:uspor}, the only possibilities for the triple of letters
\[
\Big({\rm Fst}\big(\eta(A)\big),{\rm Fst}\big(\eta(B)\big),{\rm Fst}\big(\eta(C)\big)\Big)\quad\text{ and }\quad
\Big({\rm Lst}\big(\eta(A)\big),{\rm Lst}\big(\eta(B)\big),{\rm Lst}\big(\eta(C)\big)\Big)
\]
are $(A,A,B)$, $(A,B,B)$, $(B,B,C)$, and $(B,C,C)$.

We will prove the following claim: Let $\varphi,\psi$ be the pair of
amicable Sturmian morphisms over the alphabet $\{0,1\}$ such that
$\eta={\rm ter}(\varphi,\psi)$.
\begin{itemize}
\item[(i)]
If $\eta=\eta_L$, i.e., $\eta$ is the leftmost conjugate of itself, then either
\[
\begin{aligned}
\psi&=\psi_L\quad\text{and}\quad
\Big({\rm Fst}\big(\eta_L(A)\big),{\rm Fst}\big(\eta_L(B)\big),{\rm Fst}\big(\eta_L(C)\big)\Big)=(A,B,B)\,,\quad\text{or}\\
\varphi&=\varphi_L\quad\text{and}\quad
\Big({\rm Fst}\big(\eta_L(A)\big),{\rm Fst}\big(\eta_L(B)\big),{\rm Fst}\big(\eta_L(C)\big)\Big)=(B,B,C)\,.
\end{aligned}
\]
\item[(ii)]
If $\eta=\eta_R$, i.e., $\eta$ is the rightmost conjugate of itself, then either
\[
\begin{aligned}
\varphi&=\varphi_R\quad\text{and}\quad
\Big({\rm Lst}\big(\eta_L(A)\big),{\rm Lst}\big(\eta_L(B)\big),{\rm Lst}\big(\eta_L(C)\big)\Big)=(A,B,B)\,,\quad\text{or}\\
\psi&=\psi_R\quad\text{and}\quad
\Big({\rm Lst}\big(\eta_L(A)\big),{\rm Lst}\big(\eta_L(B)\big),{\rm Lst}\big(\eta_L(C)\big)\Big)=(B,B,C)\,.
\end{aligned}
\]
\end{itemize}
In order to prove (i), let us discuss the case $\eta=\eta_L$ and
${\rm Fst}\big(\eta(A)\big)=A$, ${\rm Fst}\big(\eta(C)\big)=B$.
Since $\eta(A)={\rm ter}\big(\varphi(0),\psi(0)\big)$, necessarily
${\rm Fst}\big(\varphi(0)\big)={\rm Fst}\big(\psi(0)\big)=0$. As,
$\eta(C)={\rm ter}\big(\varphi(1),\psi(1)\big)$, necessarily ${\rm
Fst}\big(\varphi(1)\big)=0$ and ${\rm Fst}\big(\psi(1)\big)=1$.
Thus, the first letter of $\eta(B)={\rm
ter}\big(\varphi(01),\psi(10)\big)$ is $B$. Therefore the triple
$(A,A,B)$ is excluded. Moreover, we see that $\psi=\psi_L$. By the
same reasoning, we proceed in the case that ${\rm
Fst}\big(\eta(A)\big)=B$, ${\rm Fst}\big(\eta(C)\big)=C$ to exclude
the triple $(B,C,C)$ and prove $\varphi=\varphi_L$. The proof of
(ii), i.e., the case $\eta=\eta_R$ is analogous.

Consider $\big({\rm Fst}\big(\eta_L(A)\big),{\rm
Fst}\big(\eta_L(B)\big),{\rm Fst}\big(\eta_L(C)\big)\big) = (A,B,B)$
and $\psi=\psi_L$. If $\xi_1\triangleleft\cdots\triangleleft\xi_N$
are Sturmian morphisms of~\eqref{eq:sturmorf} with the same
incidence matrix, then we have $\psi=\xi_1$, and $\varphi=\xi_j$ for
some $1<j\leq N$. Consider now the substitution $\eta_R$ and denote
$\varphi'$, $\psi'$ the amicable Sturmian morphisms such that
$\eta_R = {\rm ter}\big(\varphi',\psi'\big)$. By item (ii), either
$\varphi'$ or $\psi'$ is equal to $\xi_N$. Due to
Lemma~\ref{l:sturmconjug}, we know that $\varphi'=\xi_N$, whence by
item (ii), the substitution $\eta_R$ satisfies $\big({\rm
Lst}\big(\eta_R(A)\big),{\rm Lst}\big(\eta_R(B)\big),{\rm
Lst}\big(\eta_R(C)\big)\big)=(A,B,B)$.

The case $\big({\rm Fst}\big(\eta_L(A)\big),{\rm Fst}\big(\eta_L(B)\big),{\rm Fst}\big(\eta_L(C)\big)\big) = (B,B,C)$ is treated similarly.
\pfk

\begin{corollary}\label{c:intercepty}
Let $\eta$ be a primitive substitution given by Theorem~\ref{thm:invhom} fixing a 3iet word ${\bf u}_\rho$.
If $\eta$ satisfies $\big({\rm Fst}\big(\eta(A)\big),{\rm Fst}\big(\eta(B)\big),{\rm Fst}\big(\eta(C)\big)\big) = (A,B,B)$, then $\rho=\alpha$, and if it satisfies
$\big({\rm Fst}\big(\eta(A)\big),{\rm Fst}\big(\eta(B)\big),{\rm Fst}\big(\eta(C)\big)\big) = (B,B,C)$, then $\rho=\beta$.
\end{corollary}

\pfz Let $I$ be the interval corresponding to $\eta$ such that $T_I$
is homothetic to $T$.
Denote $I_X=\{x\in I \colon R_I(x)=X\}$. If $\big({\rm
Fst}\big(\eta(A)\big),{\rm Fst}\big(\eta(B)\big),{\rm
Fst}\big(\eta(C)\big)\big) = (A,B,B)$, then the boundary between
intervals $I_A$ and $I_B$, i.e., the discontinuity point of $T_I$,
is equal to the point $\alpha$. Since $T_I$ is homothetic to $T$,
the homothety map $\Phi$ maps the discontinuity points of $T$ to the
discontinuity points of $T_I$, i.e., $\Phi(\alpha)=\alpha$. Since
the fixed point of the homothety is equal to the intercept of the
infinite word coded by $\eta$, we have $\rho=\alpha$. The second
implication is analogous. \pfk

As a byproduct of our results, it is possible, for a given substitution $\xi$ admitting a non-degenerate 3iet word ${\bf u}$ as a fixed point, to give a formula allowing to determine the parameters of ${\bf u}$, i.e., the parameters $\alpha, \beta$ of the transformation $T$ and
the  intercept $\rho$ such that ${\bf u}={\bf u}_\rho$ is a coding of $\rho$ under $T$.
Similar formula for Sturmian morphisms, i.e., those having some word coding an exchange of two intervals as a fixed point, has been given in~\cite{Peng}.

The identification of the parameters $\alpha$ and $\beta$ of the 3iet
$T$ is a straightforward task: The values $\alpha$, $\beta -\alpha$, and
$1-\beta$ are frequencies of the letters $A$, $B$ and $C$, respectively, in any infinite word coding some orbit of $T$.
Moreover, the frequencies of letters of a fixed point of a primitive substitution form can be easily determined from the eigenvector corresponding to the dominant eigenvalue of the incidence matrix of the substitution, see \cite{queffelec}.

Therefore, the only nontrivial task is to determine the intercept $\rho$.
For this purpose we use the substitution $\eta$ assigned to the substitution $\xi$ by Theorem \ref{thm:invhom}  and its leftmost conjugate $\eta_L$.
The substitution $\eta$ has exactly one eigenvalue belonging to the interval $(0,1)$, see comments below Theorem \ref{thm:invhom}.
Let this eigenvalue be denoted by $\lambda$.
Let $w$ be the word of conjugacy between $\eta$ and $\eta_L$, i.e., $\eta(a)w=w\eta_L(a)$ for any $a \in \{A,B,C\}$.
Recall that symbols $|w|_A$, $|w|_B$, and $|w|_C$ stand for the number of the letters $A$, $B$, and $C$ occuring in $w$.

\begin{theorem}\label{t:interceptconjug}
Let $\xi:\{A,B,C\}^*\to\{A,B,C\}^*$ be a primitive substitution such that it has a fixed point ${\bf u}$.
Suppose that  ${\bf u} $ is a coding  of  an orbit of a point, say $\rho$, under a non-degenerate 3iet $T$  with parameters $\alpha$ and $\beta$. Let $\lambda$, $\eta$, $\eta_L$ and $w$ be as above.
We have
\[
\rho = \rho_L + \frac1{1-\lambda}(1-\alpha,1-\alpha-\beta,-\beta)\left(\begin{smallmatrix}|w|_A\\|w|_B\\|w|_C\end{smallmatrix}\right)\,,
\]
where
$\rho_L=\alpha$ if $\eta_L(A)$ starts with $A$  and $\rho_L=\beta$ if $\eta_L(A)$ starts with $B$.
\end{theorem}

\pfz
According to Lemma~\ref{l:sedispravne}, $\eta_L(A)$ starts in $A$ or $B$. Denote by $\rho_L$ the intercept of the 3iet word fixed by $\eta_L$. By Corollary~\ref{c:intercepty},
$\rho_L$ is equal to $\alpha$ if $\eta_L(A)$ starts in $A$ and it is equal to $\beta$ if $\eta_L(A)$ starts in $B$.
We can use Proposition~\ref{p:conjug} to derive
\[
(1-\lambda)\rho_L = T^{n}\big((1-\lambda)\rho\big)\,,\quad\text{ where } n=|w|.
\]
The definition of the transformation $T$ implies  the following observation:
If $w$ of lenght $n$  is a prefix of  ${\bf u}_{x}$, then   ${\bf u}_{x}=w{\bf u}_{T^n(x)}$  and
\[
T^n(x)=x+(1-\alpha,1-\alpha-\beta,-\beta)\left(\begin{smallmatrix}|w|_A\\|w|_B\\|w|_C\end{smallmatrix}\right)\,.
\]
Combining these two facts with Remark~\ref{pozn:intercept}, we get
\[
(1-\lambda)\rho = (1-\lambda)\rho_L + (1-\alpha,1-\alpha-\beta,-\beta)\left(\begin{smallmatrix}|w|_A\\|w|_B\\|w|_C\end{smallmatrix}\right)\,.
\]
The statement follows.
\pfk

%%%%%%%%%%%%%%%%%%%%%%%%%%%%%%%%%%%%%%%%%%%%%%%%%%%%%%%%%%%%%%%%%%%%%%%%%%
\section{Applications}\label{sec:HKS}
%%%%%%%%%%%%%%%%%%%%%%%%%%%%%%%%%%%%%%%%%%%%%%%%%%%%%%%%%%%%%%%%%%%%%%%%%%
\subsection{Class $P$ conjecture for non-degenerate 3iet}

This subsection is devoted to a question coming from another field, namely mathematical physics, where notions from combinatorics on words appear naturally in the study of the spectra of Schr\"{o}dinger operators associated to infinite sequences.
The question is stated in an article of Hof, Knill and Simon \cite{HoKnSi} and concerns infinite sequences generated by a substitution over a finite alphabet.
The authors show in their paper that if a sequence contains infinitely many palindromic factors (such sequences are called \textit{palindromic}), then the associated operator has a purely singular continuous spectrum.
In the same paper, the following  class of substitutions is defined.

\begin{definition}
Let $\varphi$ be a substitution over an alphabet $\A$. We say that $\varphi$ belongs to the \textit{class $P$} if there exists a palindrome
$p$ such that for every $a\in\A$ one has $\varphi(a)=pp_a$ where $p_a$ is a palindrome.
We say that $\varphi$ is of \textit{class $P'$} if it is conjugate to some morphism in class $P$.
\end{definition}

Hof, Knill and Simon  ask the following question: ``Are there
(minimal) sequences containing arbitrarily long palindromes that
arise from substitutions none of which belongs to class $P$?'' A
discussion on how to transform this question into a mathematical
formalism can be found in \cite{LaPe14}.

The first result concerning class $P$ was given by Tan in
\cite{BoTan}. The author extended class $P$ by morphisms conjugated
to the elements of class $P$, since it is well-known that fixed
points of conjugated morphisms have the same set of factors. This
extended class is denoted by $P'$.

The conjecture, stemming from the question of Hof, Knill and Simon, states that every pure morphic (uniformly recurrent) palindromic sequence is a fixed point of a morphism of class $P'$.
It is referred to as `class $P$ conjecture'.

In~\cite{BoTan}, it is shown that
if a fixed point of a primitive substitution $\varphi$ over a binary
alphabet is palindromic, then
 the substitution $\varphi$ or $\varphi^2$ belongs to class $P'$.
In \cite{La2013}, Labb\'e shows that the assumption of a binary alphabet in the theorem of Tan is essential.
He shows that the fixed point of the substitution
\[
A\mapsto ABA,\ B\mapsto C,\ C\mapsto BAC,
\]
is palindromic. The substitution clearly does not belong to class $P'$. Moreover, no other substitution fixing the same infinite word belongs to class $P'$. Is is easy to see that Labb\'e's substitution fixes a degenerate 3iet word, namely 3iet word coding
  the orbit of $\rho=\frac{2-\sqrt{2}}{4}$ under the 3iet with parameters $\alpha =\frac12$ and $\beta=\frac{3-\sqrt{2}}{2}$.

We show that a ternary analogue of the theorem of Tan holds in the
context of codings of non-degenerate 3iet with the permutation $(321)$. The following
lemma is a generalization of a result obtained for binary alphabets
by Tan~\cite{BoTan}, also shown in \cite{LaPe14}. We provide a
different proof.

\begin{proposition}\label{p:P'}
Let $\varphi: \A \to \A$ be a non-erasing morphism. The morphism $\varphi$ is conjugate to $\overline{\varphi}$ if and only if $\varphi$ is of class $P'$.
\end{proposition}

\pfz
$(\Leftarrow)$:
Since $\varphi$ is of class $P'$, there exists a morphism $\varphi'$ of class $P$ which is conjugate to $\varphi$, i.e., there exists a word $w$ such that
$w\varphi(a) = \varphi'(a)w$ or $\varphi(a)w = w\varphi'(a)$ for every letter $a$.

We can suppose that $w\varphi(a) = \varphi'(a)w$ for every letter $a$ as the other case is analogous.
It implies $\varphi(a)  = w^{-1}pp_aw$ for some palindromes $p_a$ and $p$.
Thus, $\overline{\varphi(a)} = \overline{w} p_a p (\overline{w})^{-1}$ for every letter $a$.
In other words, the morphism $\overline{\varphi}$ is conjugate to $\overline{\varphi'}$.
Since $\overline{\varphi'}$ is clearly conjugate to $\varphi$, we conclude that $\varphi$ is conjugate to $\overline{\varphi}$.

$(\Rightarrow)$:
Since $\varphi$ is conjugate to $\overline{\varphi}$, there exists a word $w \in \B^*$ such that for every $a \in \A$, we have
\[
 \varphi(a)w = w\overline{\varphi(a)} \quad \text{  or  } \quad w\varphi(a) = \overline{\varphi(a)}w\,.
\]
Suppose first that $\varphi(a)w = w\overline{\varphi(a)}$ holds.
By Lemma~1 in~\cite{BlBrLa}, this implies that $w$ is a palindrome.
Let $u \in \A^*$ and $c \in \{\varepsilon\} \cup \A$ be such that $w = uc\overline{u}$.
We can thus write
\[
\varphi(a)uc\overline{u} = uc\overline{u}\overline{\varphi(a)}\,.
\]
By applying $(uc)^{-1}$ from the left and $(c\overline{u})^{-1}$ from the right, we obtain for any $a\in\A$
\[
c^{-1}u^{-1}\varphi(a)u = \overline{u}\overline{\varphi(a)}\overline{u}^{-1}c^{-1} = \overline{c^{-1}u^{-1}\varphi(a)u}\,.
\]
This means that the word $p_a:= c^{-1}u^{-1}\varphi(a)u$ is a palindrome.
Set $p:=c$.
Denote by $\varphi'$ the morphism defined for all $a \in \A$ by $\varphi'(a)=pp_a = u^{-1}\varphi(a)u$.
Obviously, $\varphi$ is conjugate to $\varphi'$ which is of class $P$. Therefore $\varphi\in P'$.

The case $w\varphi(a) = \overline{\varphi(a)}w$ is analogous.
\pfk

We are now in position to complete the proof the main theorem.

\begin{theorem}\label{thm:hks}
If $\xi$ is a primitive substitution fixing a non-degenerate 3iet word, then $\xi$ or $\xi^2$ belongs to class $P'$.
\end{theorem}

\pfz
Denote by $\eta\in\{\xi,\xi^2\}$ the substitution from Theorem~\ref{thm:invhom}.
There exist intervals $I_L$ and $I_R\subset[0,1)$ such that $\eta_L(A),\eta_L(B)$ and $\eta_L(C)$ are the $I_L$-itineraries,
$\eta_R(A),\eta_R(B)$ and $\eta_R(C)$ are the $I_R$-itineraries, and such that $T_{I_L}$ and $T_{I_R}$ are 3iets homothetic to $T$.

Lemma~\ref{l:sedispravne} implies that
\[
\Big({\rm Fst}\big(\eta_L(A)\big),{\rm Fst}\big(\eta_L(B)\big),{\rm Fst}\big(\eta_L(C)\big)\Big) \! \! = \! \! \Big({\rm Lst}\big(\eta_R(A)\big),{\rm Lst}\big(\eta_R(B)\big),{\rm Lst}\big(\eta_R(C)\big)\Big)
\]
and this triple of letters equals $(A,B,B)$ or $(B,B,C)$.
Suppose it is equal to $(A,B,B)$.
Note that by Corollary~\ref{c:intercepty}, $\eta_L$ fixes the infinite word ${\bf u}_\alpha$.

According to Proposition~\ref{p:mirror}, the induced transformation
$T_{\overline{I_R}}$ is again homothetic to $T$ and the
corresponding substitution is $\overline{\eta_R}$. Since it is the
mirror substitution to $\eta_R$, we have $\Big({\rm
Fst}\big(\overline{\eta_R}(A)\big),{\rm
Fst}\big(\overline{\eta_R}(B)\big),{\rm
Fst}\big(\overline{\eta_R}(C)\big)\Big)=(A,B,B)$. By
Corollary~\ref{c:intercepty}, the substitution $\overline{\eta_R}$
also fixes the infinite word ${\bf u}_\alpha$. Since the intervals
$I_L$ and $\overline{I_R}$ are of the same length and are homothetic
to the interval $[0,1)$ with the same homothety center $\alpha$,
necessarily $I_L=\overline{I_R}$ and thus
$\overline{\eta_R}=\eta_L$. Consequently, $\eta_R$ is conjugate to
its mirror image. We apply Proposition~\ref{p:P'} to finish the
proof.

In case that $\left({\rm Fst}\big(\eta_L(A)\big),{\rm
Fst}\big(\eta_L(B)\big),{\rm Fst}\big(\eta_L(C)\big)\right) =
(B,B,C)$, we proceed in a similar way. In this case, the center of
the homothety of the intervals $I_L = \overline{I_R}$ and $[0,1)$ is
$\beta$. \pfk

Let us mention that another analogue of Tan's result is already known for marked morphisms.
Recall that a substitution $\xi$ over an alphabet $\A$ is called \textit{marked} if its leftmost conjugate $\xi_L$ and its rightmost conjugate $\xi_R$ satisfy
\[
{\rm Fst}\big(\xi_L(a)\big)\neq {\rm Fst}\big(\xi_L(b)\big) \quad \text{ and } \quad {\rm Lst}\big(\xi_R(a)\big)\neq {\rm Lst}\big(\xi_R(b)\big)
\]
for distinct $a,b\in\A$. It can be shown that if $\xi$ is marked,
then all its powers are marked. In \cite{LaPe14}, it is shown that
for a marked morphism $\xi$ with fixed point ${\bf u}$ having
infinitely many palindromes, some power $\xi^k$ belongs to class
$P'$.

Our Lemma~\ref{l:sedispravne} shows that a substitution fixing a non-degenerated 3iet word cannot be marked. Theorem~\ref{thm:hks} thus
provides a new family of substitutions satisfying class $P$ conjecture.

%%%%%%%%%%%%%%%%%%%%%%%%%%%%%%%%%%%%%%%%%%%%%%%%%%%%%%%%%%%%%%%%%%%%%%%%%%
\subsection{Properties of 3iet preserving substitutions}

A morphism which maps a 3iet word
to a 3iet word is called 3iet preserving. Morphisms which preserve
2iet words, i.e. Sturmian words, are called Sturmian and they have been
extensively studied for many years. Sturmian morphisms form a
monoid which is generated by three morphisms only \cite{MiSeSturmianMorphisms}. In
contrary to Sturmian morphisms, the class of 3iet preserving
morphisms is not completely described. Only partial results are
known. For example, the monoid of  3iet preserving morphisms  is
not finitely generated  \cite{AmHaPe}  and contains the ternarizations  we
defined in Section~\ref{sec:ternarizace}, see
\cite{AmMaPe}. Our  previous considerations lead to some comments on
properties of 3iet preserving substitution.

\begin{itemize}
\item
Our results of Sections~\ref{sec:3} and~\ref{sec:substitutions} allow us to say more about the structure of substitutions fixing 3iet words.

\begin{corollary} \label{c:strukturamorfizmu}
Let $\eta$ be a primitive substitution of Theorem~\ref{thm:invhom} fixing a non-degenerate 3iet word over the alphabet $\{A,B,C\}$. We have
\[
\eta(B)=\omega_{AC\to B}\big(\eta(AC)\big)=\omega_{CA\to B}\big(\eta(CA)\big)
\quad\text{or}\quad
\eta(B)=\omega_{B\to CA}\big(\eta(AC)\big)=\omega_{B\to AC}\big(\eta(CA)\big)\,.
\]
\end{corollary}

\pfz By Theorem~\ref{thm:invhom}, $\eta$ corresponds to an interval
$I$ such that $T_I$ is homothetic to $T$. Since $T$ is
non-degenerate, also  $T_I$ is non-degenerate, and therefore its
discontinuity points $\c,\d$ are distinct. By
Proposition~\ref{p:3itinerare}, the three $I$-itineraries are of the
form given by cases (i) or (ii). \pfk

\begin{example}
We can illustrate the above corollary on substitutions from Example~\ref{ex:ternarizacesubst}.
We have
\[
\begin{aligned}
\varphi(0)&=0110101\\
\varphi(1)&=01101
\end{aligned}\,,
\qquad
\begin{aligned}
\psi(0)&=1010101\\
\psi(1)&=10101
\end{aligned}\,,
\]
and
\[
\eta(A)=BCACAC,\quad \eta(B)=BCACBBCAC,\quad \eta(C)=BCAC.
\]
We can check that $\eta$ satisfies the property given in Corollary~\ref{c:strukturamorfizmu}, namely that
\[
\begin{aligned}
\eta(B)=BCACBBCAC&=\omega_{AC\to B}\big(\eta(AC)\big)=\omega_{AC\to B}\big(BCACACBCAC\big) = \\
&= \omega_{CA\to B}\big(\eta(CA)\big) = \omega_{CA\to B}\big(BCACBCACAC\big)\,.
\end{aligned}
\]
\end{example}

The above corollary implies a relation of numbers of occurrences of letters in letter images of $\eta$ which may be used to get the following relation:
% As a consequence of the above corollary, we have for the columns of the incidence matrix $M_\eta$ we therefore have
\[
M_\eta\left(\begin{smallmatrix}~1\\-1\\~1\end{smallmatrix}\right) =
 \left(\begin{smallmatrix}|\eta(A)|_A\\|\eta(A)|_B\\|\eta(A)|_C\\\end{smallmatrix}\right)
-
\left(\begin{smallmatrix}|\eta(B)|_A\\|\eta(B)|_B\\|\eta(B)|_C\end{smallmatrix}\right)
+
\left(\begin{smallmatrix}|\eta(C)|_A\\|\eta(C)|_B\\|\eta(C)|_C\end{smallmatrix}\right)
=  \pm
\left(\begin{smallmatrix}~1\\-1\\~1\end{smallmatrix}\right)\,.
\]
Thus, $(1,-1,1)^\top$ is an eigenvector of the incidence matrix of $\eta$ corresponding to the eigenvalue $1$ and $-1$, respectively.
This fact has been already derived in~\cite{AmMaPeMatice} by other methods.

\item
One can ask whether a substitution $\eta$ fixing a 3iet word can be
in the same time the leftmost and rightmost conjugate of itself,
i.e. $\eta=\eta_L=\eta_R$. It can be easily seen that such a
situation never occurs for non-degenerate 3iets. Indeed, the proof
of Theorem~\ref{thm:hks} implies that for any primitive substitution
$\eta$ fixing a non-degenerate 3iet, $\eta_L$ and
$\overline{\eta_R}$ fix the same infinite word ${\bf u}_\rho$, where
$\rho\in\{\alpha,\beta\}$. If, moreover, $\eta_R=\overline{\eta_R}$,
then by Proposition~\ref{p:mirror} we have $1-\rho=\rho$. This
implies that $\rho=\frac12\in\{\alpha,\beta\}$. However, this cannot
happen for a non-degenerate 3iet $T$.

\item
It can be observed from Lemma~\ref{l:sedispravne} that given a
substitution $\eta$ fixing a 3iet word, its leftmost conjugate
$\eta_L$ has always two fixed points, namely either
$\lim_{n\to\infty}\eta_L^n(A)$ and $\lim_{n\to\infty}\eta_L^n(B)$,
or $\lim_{n\to\infty}\eta_L^n(B)$ and
$\lim_{n\to\infty}\eta_L^n(C)$. One can show that one of these fixed
points is a coding of the point $\rho=\alpha$ or $\beta$,
respectively, under a 3iet $T$. The other fixed point is a coding of
the same point, but under an exchange of three intervals defined
over $(0,1]$, where all the intervals are of the form
$(\cdot,\cdot]$.

\begin{example}
Consider the substitution $\eta$ from Example~\ref{ex:ternarizacesubst}. We have
\[
\eta_L(A)=ACBCAC,\quad \eta_L(B)=BBCACBCAC,\quad \eta_L(C)=BCAC.
\]
This substitution has two fixed points, namely
\[
\begin{aligned}
\lim_{n\to\infty}\eta_L^n(A) &= ACBCACBCACBBCACBCACBCACACBCACBCAC\cdots,\\
\lim_{n\to\infty}\eta_L^n(B) &= BBCACBCACBBCACBCACBCACACBCACBCAC\cdots.
\end{aligned}
\]
It can be verified that the two infinite words differ only by the
prefix $AC$ vs. $B$. The infinite word ${\bf u}_\alpha$, coding
$\alpha$ under the 3iet $T$ is equal to the fixed point
$\lim_{n\to\infty}\eta_L^n(B)$.
\end{example}

\end{itemize}

%%%%%%%%%%%%%%%%%%%%%%%%%%%%%%%%%%%%%%%%%%%%%%%%%%%%%%%%%%%%%%%%%%%%%%%%%%
\section*{Acknowledgements}
\small Z.M. and E.P. acknowledge financial support by the Czech Science
Foundation grant GA\v CR 13-03538S, \v S.S. acknowledges financial support by the Czech Science
Foundation grant  GA\v CR 13-35273P.

\bibliographystyle{siam}
\IfFileExists{biblio.bib}{\bibliography{biblio}}{\bibliography{../!bibliography/biblio}}
\end{document}